\documentclass[10pt]{amsart}
\usepackage{amsmath}
\usepackage{amssymb}
\usepackage{amsfonts}
\usepackage{amsthm}
\usepackage{enumerate}
\usepackage[all]{xy}
\usepackage[latin1]{inputenc}
\usepackage{mathdots}
\usepackage{graphicx}
\usepackage{latexsym}
\input xy

\newtheorem{Theo}{Theorem}[section]
\newtheorem{Prop}[Theo]{Proposition}
\newtheorem{Cor}[Theo]{Corollary}
\newtheorem{Lemma}[Theo]{Lemma}
\newtheorem{Conj}[Theo]{Conjecture}
\theoremstyle{definition}

\newtheorem{Exam}[Theo]{Example}
\newtheorem{Remark}[Theo]{Remark}

\def\mystrut(#1,#2){\vrule height #1pt depth #2pt width 0pt}
\newcommand{\rep}{{\rm rep}}
\newcommand{\rrep}{\overline{\mystrut(5,0) {\rm rep\hspace{-1pt}}}\,}
\newcommand{\rreps}{\overline{\mystrut(3.5,0) {\rm rep\hspace{-1pt}}}\,}
\newcommand{\Rep}{{\rm Rep}}
\newcommand{\Hom}{{\rm Hom}}
\newcommand{\Ext}{{\rm Ext}}

\newcommand{\N}{\mathbb{N}}
\newcommand{\Z}{\mathbb{Z}}
\newcommand{\C}{\mathcal{C}}
\newcommand{\A}{\mathbb{A}}

\def\id{\hbox{1\hskip -3pt {\sf I}}}

\def\Ga{\hbox{$\mathit\Gamma$}}

\begin{document}

\begin{abstract} Let $Q$ be a strongly locally finite quiver and denote by $\rep(Q)$ the category of locally
finite dimensional representations of $Q$ over some fixed field $k$. The main purpose of this paper is to get a better understanding of $\rep(Q)$ by means of its Auslander-Reiten quiver. To achieve this goal,
we define a category $\rrep(Q)$ which is a full, abelian and Hom-finite subcategory of $\rep(Q)$ containing all the almost
split sequences of $\rep(Q)$.  We give a complete description of the Auslander-Reiten quiver
of $\rrep(Q)$ by describing its connected components.  Finally, we
prove that these connected components are also connected components
of the Auslander-Reiten quiver of $\rep(Q)$. We end the paper by giving a conjecture describing the Auslander-Reiten components of $\rep(Q)$ that cannot be obtained as Auslander-Reiten components of $\rrep(Q)$.
\end{abstract}

\title[Auslander-Reiten quiver of an infinite quiver]{On the Auslander-Reiten quiver of the representations of an infinite quiver}

\author{Charles Paquette}
\address{Charles Paquette, Dept. of Math. and Stat., U. of New Brunswick, Fredericton, NB, E3B 5A3, Canada.}

\email{charles.paquette@usherbrooke.ca}

\maketitle

\section*{Introduction}

It follows from a result of Gabriel that any basic and finite
dimensional algebra over an algebraically closed field $k$ is given
by a quiver with relations, that is, it is a quotient of a path
algebra $kQ$ by an admissible ideal $I$ of $kQ$, where $Q$ is a
finite quiver; see, for example, \cite[Section 3, Theorem 1.9]{ARS}.
Therefore, the algebras of the form $kQ$ where $Q$ is a quiver and
$k$ is any field are of particular interest. If $Q$ is finite and
contains no oriented cycle, $kQ$ is a finite dimensional hereditary
algebra and the Auslander-Reiten theory of $kQ$ is well established;
see \cite{ASS, Ri}. In \cite{BLP}, the Auslander-Reiten theory of
$kQ$ where $Q$ is infinite but strongly locally finite is studied.
Indeed, the category $\rep^+(Q)$ consisting of the finitely
presented representations is studied by means of its
Auslander-Reiten theory. A complete description of the
Auslander-Reiten quiver of $\rep^+(Q)$ is given. There is a unique
preprojective component, some preinjective components and four types
of regular components. Dually, the category $\rep^-(Q)$ of the
finitely co-presented representations has a well-understood
Auslander-Reiten theory and its Auslander-Reiten quiver is
completely described. In \cite{Pa}, it is shown that an almost split
sequence in $\rep(Q)$ necessarily starts with a finitely
co-presented representation and ends with a finitely presented one.
Therefore, it seems that the categories $\rep^+(Q)$ and $\rep^-(Q)$ somehow control the Auslander-Reiten theory of the whole category $\rep(Q)$. The main goal of this paper is to show that this is indeed the case. One can construct a full
subcategory $\rrep(Q)$ of $\rep(Q)$ which is abelian and Hom-finite and
contains the Auslander-Reiten theory of $\rep(Q)$ (and hence the
Auslander-Reiten theory of both $\rep^+(Q)$ and $\rep^-(Q)$). We give a complete description of the Auslander-Reiten quiver of $\rrep(Q)$ and show how the knowledge of this helps one to get a partial
description of the irreducible morphisms in the whole category $\rep(Q)$. In particular, we provide all the Auslander-Reiten components of $\rep(Q)$ for which the translation acts non-trivially.

\smallskip

In the first section, we provide some background on representations
of quivers and recall some key facts concerning the categories
$\rep^+(Q), \rep^-(Q)$ and $\rep(Q)$ and the existence of almost
split sequences in $\rep(Q)$.  In Section $2$, we define and study
finite extensions, and provide some properties of these extensions.
 In Section $3$, we define a full subcategory $\rrep(Q)$ of $\rep(Q)$ which
consists of the finite extensions of objects in $\rep^-(Q)$ by those
in $\rep^+(Q)$. The additive $k$-category $\rrep(Q)$ is shown to be
abelian and Hom-finite.  We also provide a result
saying that each representation in $\rrep(Q)$ is built from a
projective representation in $\rep^+(Q)$, an injective
representation in $\rep^-(Q)$ and a finite dimensional
representation.  In Section $4$, we show how to control the domain
and co-domain of an irreducible morphism in $\rrep(Q)$. In
particular, if $M \to N$ is irreducible in $\rrep(Q)$ with both
$M,N$ indecomposable, then either $M \in \rep^-(Q)$ or $N \in
\rep^+(Q)$. In Section $5$, we give a complete description of the
Auslander-Reiten quiver of $\rrep(Q)$.  If $Q$ is connected, there
is a unique preprojective component, a unique preinjective component
and four possible types of regular components. Finally, in Section
$6$, we describe partially the Auslander-Reiten quiver of $\rep(Q)$,
by showing that the connected components of the Auslander-Reiten
quiver of $\rrep(Q)$ are connected components of the
Auslander-Reiten quiver of $\rep(Q)$. We propose a conjecture for
the shapes of the other connected components of the Auslander-Reiten
quiver of $\rep(Q)$.

\section{Representations of quivers}

Let $Q=(Q_0,Q_1)$ be a \emph{strongly locally finite} quiver, that
is, a locally finite quiver for which the number of paths between
any given pair $x,y$ of vertices of $Q$ is finite. This last
property will be referred to as $Q$ being \emph{interval-finite}.
Throughout the paper, we fix $k$ to be any field. Recall that a
representation $M$ of $Q$ (over $k$) is defined by the following
data. For each $x \in Q_0$, $M(x)$ is a $k$-vector space and for
each arrow $\alpha: x \to y \in Q_1$, $M(\alpha) : M(x) \to M(y)$ is
a $k$-linear map.  If $M,N$ are representations in $\Rep(Q)$, then a
map $f: M \to N$ is a family $\{f_x: M(x) \to N(x) \mid x \in Q_0\}$
of $k$-linear maps such that for any arrow $\alpha : x \to y$ in
$Q$, we have $f_yM(\alpha)=N(\alpha)f_x$. The category of all
representations of $Q$ over $k$ is denoted by $\Rep(Q)$. Indeed,
$\Rep(Q)$ is the category of all $k$-linear covariant functors from
the path category $kQ$ to the category of all $k$-vector spaces. The
importance of the category $\Rep(Q)$ relies on the fact that it is
equivalent to the category ${\rm Mod}A$ of all unitary left
$A$-module, where $A$ is the path algebra $kQ$ (and has no identity
if $Q$ is infinite). Here, a left $A$-module $M$ is \emph{unitary} if $AM=M$. A representation $M \in \Rep(Q)$ is said to be
\emph{locally finite dimensional} if $M(x)$ is finite dimensional
for all $x \in Q_0$; and \emph{finite dimensional} if
$\textstyle{\sum}_{x \in Q_0}M(x)$ is finite dimensional. The full
subcategory of $\Rep(Q)$ of all locally finite dimensional
representations of $Q$ is denoted by $\rep(Q)$. In some sense, the
objects in $\rep(Q)$ are close from being finite dimensional and
they have nice properties. For example, every indecomposable object
in $\rep(Q)$ has a local endomorphism algebra; see \cite{GR}, and
there exists a (pointwise) duality $D_Q: \rep(Q) \to \rep(Q^{\rm
\,op})$ where $Q^{\rm \,op}$ is the opposite quiver of $Q$; see
\cite{BLP}. It is also shown in \cite[Section 3.6]{GR} that
$\Rep(Q)$ (and $\rep(Q)$) is abelian and hereditary, that is,
$\Ext_{\Rep(Q)}^2(-,-)$ vanishes. Note however that, in general,
$\rep(Q)$ is not Hom-finite.

\smallskip

Let us now introduce two important families of representations in
$\rep(Q)$. First, given $x,y \in Q_0$, let us denote by $Q(x,y)$ the
(finite) set of paths from $x$ to $y$ in $Q$. For $a \in Q_0$, let
$P_a$ denote the representation defined as follows. For $x \in Q_0$,
$P_a(x) = k\langle Q(a,x) \rangle$ and for an arrow $\alpha: x \to
y$, $P_a(\alpha)$ is the right multiplication by $\alpha$. Since $Q$
is interval-finite, $P_a \in \rep(Q)$. It is easy to show that $P_a$
is indecomposable projective in $\rep(Q)$; see, for example,
\cite[Proposition 1.3]{BLP}. Dually, for $a \in Q_0$, denote by
$I_a$ the following representation. For $x \in Q_0$, $I_a(x) =
k\langle Q(x,a)\rangle$ and for an arrow $\alpha: x \to y$,
$I_a(\alpha)$ is the transpose of the map $kQ(y,a) \to kQ(x,a)$
which is the left multiplication by $\alpha$. Since $Q$ is
interval-finite, $I_a \in \rep(Q)$. Moreover, using the duality $D_Q$, we see that $I_a$ is
indecomposable injective in $\rep(Q)$; see also \cite{BLP}.

\smallskip

The full subcategory of $\Rep(Q)$ whose objects are the finitely
presented (finitely co-presented, respectively) representations is
denoted by $\rep^+(Q)$ ($\rep^-(Q)$, respectively).  Since $Q$ is
strongly locally finite, $\rep^+(Q)$ and $\rep^-(Q)$ are indeed full
subcategories of $\rep(Q)$. Since $\rep(Q)$ is hereditary, $M \in
\rep^+(Q)$ if and only if there exists a short exact sequence
$$0 \to \textstyle{\bigoplus}_{i=s+1}^r P_{x_i} \to \textstyle{\bigoplus}_{i=1}^s P_{x_i} \to M \to 0$$
in $\rep(Q)$ where the $x_i$ are vertices in $Q$.
Similarly, $M \in \rep^-(Q)$ if and only if there exists a short
exact sequence
$$0 \to M \to \textstyle{\bigoplus}_{i=1}^{s'} I_{y_i} \to \textstyle{\bigoplus}_{i=s'+1}^{r'} I_{y_i} \to 0$$
in $\rep(Q)$ where the $y_i$ are vertices in $Q$. The
indecomposable projective representations in $\rep^+(Q)$, up to
isomorphisms, are the $P_x$ for $x \in Q_0$. Similarly, the
indecomposable injective representations in $\rep^-(Q)$, up to
isomorphisms, are the $I_x$ for $x \in Q_0$.  However, $\rep(Q)$ may
have indecomposable projective representations which are not
isomorphic to the $P_x$, and dually, may have indecomposable
injective representations which are not isomorphic to the $I_x$.

\smallskip

Let us now introduce more definitions. For this purpose, fix $M \in
\rep(Q)$. The \emph{support} of $M$, written as ${\rm supp}(M)$, is
the full subquiver of $Q$ generated by the vertices $x \in Q_0$ for
which $M(x) \ne 0$. If $M \in \rep^+(Q)$, then it follows from the
definition of $\rep^+(Q)$ and the fact that $Q$ is strongly locally
finite that ${\rm supp}(M)$ is \emph{top-finite}, that is, contains
finitely many source vertices and every vertex $x$ in ${\rm
supp}(M)$ is a successor of one such source vertex. Dually, if $M
\in \rep^-(Q)$, then ${\rm supp}(M)$ is \emph{socle-finite}, that
is, contains finitely many sink vertices and every vertex $x$ in
${\rm supp}(M)$ is a predecessor of one such sink vertex. Observe
that a socle-finite quiver may contain a \emph{left-infinite path},
that is, a path of the form
$$\cdots \to \circ \to \circ \to \circ \to \circ$$
but does not contain a \emph{right-infinite path}, that is, a path
of the form
$$\circ \to \circ \to \circ \to \circ \to \cdots.$$
Dually, a top-finite quiver may contain a right-infinite path but
does not contain left-infinite paths. It is easy to see that a full
subquiver of $Q$ which is top-finite and socle-finite needs to be
finite.  In particular, $$\rep^b(Q):=\rep^+(Q) \cap \rep^-(Q)$$
consists of all the finite dimensional representations of $Q$. If
$x$ is a vertex in $Q$, then one says that $M$ is \emph{supported}
by $x$ if $M(x) \ne 0$, or equivalently, if $x \in
{\rm supp}(M)$. 
Finally, if $\alpha \in Q_1$, we say that $M$
is \emph{supported} by $\alpha$ 
if $M(\alpha) \ne 0$. Note that if $\alpha$ is an arrow in ${\rm
supp}(M)$, then it does not necessarily mean that $M$ is supported
by $\alpha$.

\smallskip

In this paper, we shall use many basic results concerning
$\rep^+(Q)$ and $\rep^-(Q)$ appearing in \cite{BLP}.  This is our
main reference.  We shall, however, recall the main definitions and
results we need at the appropriate place in the sequel. Here are
some of them, that we will use freely. Both $\rep^+(Q)$ and
$\rep^-(Q)$ are Hom-finite hereditary abelian
$k$-categories.  For any indecomposable non-projective
representation $X$ in $\rep^+(Q)$, there exists an almost split
sequence
$$0 \to X' \to E \to X \to 0$$
in $\rep(Q)$ where $X' \in \rep^-(Q)$.  This sequence is almost
split in $\rep^-(Q)$ if and only if $X$ is finite dimensional.
Dually, for any indecomposable non-injective representation $Y$ in
$\rep^-(Q)$, there exists an almost split sequence
$$0 \to Y \to E' \to Y' \to 0$$
in $\rep(Q)$ where $Y' \in \rep^+(Q)$.  This sequence is almost
split in $\rep^+(Q)$ if and only if $Y$ is finite dimensional.


\section{Finite extensions and their properties}

In this section, we define the main category (a subcategory of
$\rep(Q)$) of interest of the paper. As we will see, this category
contains all the almost split sequences of $\rep(Q)$ and has a nice
Auslander-Reiten theory. Let $L,M,N \in \rep(Q)$. An
\emph{extension} of $N$ by $L$ is a short exact sequence
$$0 \to L \to M \to N \to 0$$
in $\rep(Q)$, where $M$ is called an \emph{extension-representation}
of $N$ by $L$. Two extensions $\eta, \gamma$ of $N$ by $L$ are
equivalent if there exists a commutative diagram
$$\xymatrixrowsep{15pt}\xymatrixcolsep{15 pt}\small\xymatrix{\eta\,:\;\; 0 \ar[r] &L \ar[r]\ar@{=}[d]& M \ar[r]\ar[d] & N \ar[r]\ar@{=}[d] & 0 \\
\gamma\,: \;\;0 \ar[r] &L \ar[r]& M' \ar[r] & N \ar[r] & 0 }$$
\normalsize in $\rep(Q)$. If $M$ is an extension-representation of
$N$ by $L$ such that all but a finite number of arrows supporting
$M$ are arrows supporting $L\oplus N$, then $M$ is said to be a
\emph{finite extension-representation} of $N$ by $L$, while the
corresponding short exact sequence is a \emph{finite extension} of
$N$ by $L$. Recall that $\Ext(N,L)$ is the set of all extensions of
$N$ by $L$ modulo the equivalence relation given above. It is an
abelian group under the Baer sum; see \cite{Mit}. Observe that if
$\eta,\gamma$ are equivalent extensions of $N$ by $L$, then $\eta$
is finite if and only if $\gamma$ is finite. Moreover, one can
easily check that if $\gamma,\eta \in\Ext(N,L)$ are finite
extensions, then their Baer sum $\gamma + \eta$ is also a finite
extension. Hence, the subset $\mathcal{E}(N,L)$ of $\Ext(N,L)$ of
all equivalence classes of finite extensions of $N$ by $L$ is a subgroup of $\Ext(N,L)$.
The following lemma tells us how to recognize such a finite
extension.

\begin{Lemma} \label{RecognizeFiniteExt}
Let $M$ be an extension-representation of $N$ by $L$ and $S$ be the
arrows $x \to y$ in $Q$ with $x \in {\rm supp}(N)$ and $y \in {\rm
supp}(L)$. Then $M$ is a finite extension-representation if and only
if $M(\alpha)=0$ for all but a finite number of $\alpha \in S$.
\end{Lemma}

\begin{proof}
Let $E$ be the set of all arrows in $Q$ supporting $M$ but not
supporting $L \oplus N$, and let $\alpha : x \to y$ with $\alpha \in
E$. If $x \in {\rm supp}(L) \backslash {\rm supp}(N)$, then
$M(x)=L(x)$. Moreover, $L$ being a sub-representation of $M$ implies
that $L(\alpha) \ne 0$, a contradiction. Hence,  $x \not \in {\rm
supp}(L) \backslash {\rm supp}(N)$. Similarly, $y \not \in {\rm
supp}(N) \backslash {\rm supp}(L)$. In particular, since ${\rm
supp}(M) = {\rm supp}(L) \cup {\rm supp}(N)$, $x \in {\rm supp}(N)$
and $y \in {\rm supp}(L)$.
\end{proof}

Let $\rrep(Q)$ be the full subcategory of $\rep(Q)$ with objects $M$
such that $M$ is a finite extension-representation of $N$ by $L$ for
some $L \in \rep^+(Q)$ and $N \in \rep^-(Q)$. We will use this category to get a better understanding of the Auslander-Reiten theory of $\rep(Q)$. Observe that Auslander has defined a similar category in \cite{Aus}, for a similar purpose but in a different setting. Given a noetherian algebra $\Lambda$, Auslander defined ${\rm arno}(\Lambda)$ as the full subcategory of the module category whose objects are the middle terms of the short exact sequences of the form $0 \to L \to M \to N \to 0$ with $L$ an artian $\Lambda$-module and $N$ a noetherian $\Lambda$-module. Note that in general, a finitely presented representation is not artinian and, more importantly, is not noetherian; see \cite{RVDB}. Hence our definition is slightly different.  When $\rep^+(Q)$ is noetherian and $\rep^-(Q)$ is artinian, then Lemma \ref{Ext+-} below implies that our category $\rrep(Q)$ contains all representations occurring as a middle term of a short exact sequence of the form $0 \to L \to M \to N \to 0$ with $L$ an artian representation and $N$ a noetherian representation.

\begin{Lemma} \label{FiniteExtLemma}
Let $M$ be a representation in  $\rrep(Q)$.  Then any short exact
sequence $0 \to L \to M \to N \to 0$ with $L \in \rep^+(Q)$ and $N
\in \rep^-(Q)$ is a finite extension.
\end{Lemma}

\begin{proof}
By assumption, there exists a finite extension
$$\xymatrixcolsep{15 pt}\xymatrix{0 \ar[r] & L' \ar[r] & M \ar[r] & N' \ar[r] & 0}$$
with $L' \in \rep^+(Q)$ and $N' \in \rep^-(Q)$. Let
$$\xymatrixcolsep{15 pt}\xymatrix{0 \ar[r] & L \ar[r] & M \ar[r] & N \ar[r] & 0}$$
be an extension with $L \in \rep^+(Q)$ and $N \in \rep^-(Q)$.
Observe that $${\rm supp}(L') \backslash {\rm supp}(L) \subseteq
{\rm supp}(N).$$  We claim that ${\rm supp}(L') \backslash {\rm
supp}(L)$ is finite.  Suppose the contrary. Since ${\rm supp}(N)$ is
socle-finite, there exists a left-infinite path $p$ in ${\rm
supp}(N)$ ending at $y \in Q_0$ such that infinitely many vertices
of $p$, say $\{x_i\}_{i \ge 1}$, lie in ${\rm supp}(L') \backslash
{\rm supp}(L)$. Since ${\rm supp}(L')$ is top-finite, infinitely
many of the $x_i$ are successor of a fixed vertex $x$ in ${\rm
supp}(L')$. Therefore, there are infinitely many vertices
$\{y_i\}_{i \ge 1}$ such that, for $i \ge 1$, there is a path from
$x$ to $y$ passing through $y_i$.  This contradicts the fact that
$Q$ is interval-finite.  Hence, ${\rm supp}(L') \backslash {\rm
supp}(L)$ is finite. Similarly, ${\rm supp}(L) \backslash {\rm
supp}(L')$ is finite. In a dual way, we can show that ${\rm supp}(N)
\backslash {\rm supp}(N')$ and ${\rm supp}(N') \backslash {\rm
supp}(N)$ are finite. Therefore, all but a finite number of arrows
starting at a vertex in ${\rm supp}(N)$ and ending at a vertex in
${\rm supp}(L)$ are arrows starting at a vertex in ${\rm supp}(N') $
and ending at a vertex in ${\rm supp}(L')$. The result follows
immediately from Lemma \ref{RecognizeFiniteExt}.
\end{proof}

\begin{Lemma} \label{Ext+-}
Let $X \in \rep^+(Q)$ and $Y \in \rep^-(Q)$.  Then $\Ext(X,Y) =
\mathcal{E}(X,Y)$.
\end{Lemma}

\begin{proof}
From Lemma \ref{RecognizeFiniteExt}, it is sufficient to show that
there are finitely many arrows starting from a vertex in ${\rm
supp}(X)$ and ending to a vertex in ${\rm supp}(Y)$. Assume the
contrary. Let $\{\alpha_i: x_i \to y_i\}_{i\ge 1}$ be an infinite
family of arrows with $x_i \in {\rm supp}(X)$ and $y_i \in {\rm
supp}(Y)$. Since ${\rm supp}(X)$ is top-finite and ${\rm supp}(Y)$
is socle-finite, there exist a source vertex $x$ in ${\rm supp}(X)$
and a sink vertex $y$ in ${\rm supp}(Y)$ such that, for infinitely
many $i \ge 1$, $x_i$ is a successor of $x$ in ${\rm supp}(X)$ and
$y_i$ is a predecessor of $y$ in ${\rm supp}(Y)$.  Hence, there are
infinitely many paths starting from $x$ and ending at $y$.  This
contradicts the fact that $Q$ is interval-finite.
\end{proof}

Recall that since $Q$ is interval-finite, $Q(x,y)$ is finite for
$x,y \in Q_0$. However, $Q(x,y)$ can be arbitrarily large when $x,y$
run through $Q_0 \times Q_0$. If there exits a global bound on
$|Q(x,y)|$, then every extension of a finitely co-presented
representation by a finitely presented one is finite.

\begin{Prop}
Suppose that $Q$ is such that there exists a positive integer $r$
such that $|Q(x,y)| \le r$ for any $x,y \in Q_0$.  Then
$\mathcal{E}(X,Y) = \Ext(X,Y)$ for any $X \in \rep^-(Q)$ and $Y \in
\rep^+(Q)$.
\end{Prop}

\begin{proof}
Let $$0 \to L \to M \to N \to 0$$ be an extension with $L \in
\rep^+(Q)$ and $N \in \rep^-(Q)$.   Let $A$ be the set of all arrows
in $Q$ starting at a vertex in ${\rm supp}(N)$ and ending at a
vertex in ${\rm supp}(L)$. It is sufficient to show that $A$ is
finite. If one of ${\rm supp(N)},{\rm supp}(L)$ is finite, then the
claim follows since $Q$ is locally finite. Hence, we may assume that
both ${\rm supp(N)},{\rm supp}(L)$ are infinite. Suppose to the
contrary that $A$ is infinite. This implies that $S=\{s(\alpha) \; |
\; \alpha \in A\}$ and $E=\{e(\alpha) \; | \; \alpha \in A\}$ are
both infinite subsets of $Q_0$.  Since ${\rm supp}(N)$ is
socle-finite and $S \subseteq {\rm supp}(N)$ is infinite, there
exists a left infinite path
$$p: \quad \cdots \to a_3 \to a_2 \to a_1$$
in ${\rm supp}(N)$ with an infinite number of $a_j$ being in $S$.
Then $J := p \cap S$ is infinite. Since ${\rm supp}(L)$ is
top-finite and
$$J'=\{e(\alpha) \mid  \alpha \in A, \; s(\alpha) \in J \} \subseteq E \subseteq {\rm supp}(L)$$ is infinite, there
exists a right-infinite path
$$b_1 \to b_2 \to b_3 \to \cdots$$
in ${\rm supp}(L)$ with an infinite number of $b_j$ being in $J'$.
It is now easy to see that the sets $Q(a_j,b_j)$, $j \ge 1$, have
arbitrary large cardinality. This is a contradiction.
\end{proof}

Let us introduce some definitions. Let $\Sigma$ be a quiver and
$\Omega$ a full subquiver of $\Sigma$. One says that $\Omega$ is
\emph{predecessor-closed} (resp. \emph{successor-closed}) in
$\Sigma$ if for any path $p: x \rightsquigarrow y$ in $\Sigma$ with
$y$ in $\Omega$ (resp. $x$ in $\Omega$), we have $x \in \Omega$
(resp. $y \in \Omega$).
We say that $\Omega$ is \emph{co-finite} in $\Sigma$ if $\Sigma_0
\backslash \Omega_0$ is finite.  If $\Omega$ is a full subquiver of
$\Sigma$, and $M$ is a representation of $\Sigma$, we denote by
$M_{\Omega}$ the restriction of $M$ to $\Omega$.  It is a
representation of $\Omega$ but can be seen as a representation of
$\Sigma$ by setting $M_{\Omega}(x)=0$ if $x \not \in \Omega$.

\begin{Lemma} \label{StandardExt}
Suppose that $M \in \rrep(Q)$. Then there exists a top-finite full
subquiver $\Omega$ of $Q$ such that $M_{\Omega} \in \rep^+(Q)$ is a
sub-representation of $M$ and $M/M_{\Omega} \in \rep^-(Q)$.  In
particular,
$$0 \to M_{\Omega} \to M \to M/M_{\Omega} \to 0$$
is a finite extension.
\end{Lemma}

\begin{proof} One
has a finite extension
$$0 \to L \to M \to N \to 0$$
with $L \in \rep^+(Q)$ and $N \in \rep^-(Q)$. Let $\Omega$ be a
top-finite and successor-closed subquiver of $Q$ supporting $L$.
Then $L$ is a sub-representation of $M_{\Omega}$. Consider the
diagram
$$\xymatrix{0 \ar[r] & L \ar[r] \ar[d] & M \ar[r] \ar@{=}[d] & N \ar[r] \ar[d]^{g} & 0 \\
0 \ar[r] & M_{\Omega} \ar[r] & M \ar[r] & M/M_{\Omega} \ar[r] & 0}$$
where $g$ is the induced morphism.  Observe that $g$ is an
epimorphism.  By the snake lemma, ${\rm Ker}\,g \cong M_{\Omega} /
L$ and hence,
$${\rm supp}({\rm Ker}\,g) \subseteq {\rm supp}(N) \cap {\rm supp}(M_{\Omega}) \subseteq {\rm supp}(N)
\cap \Omega,$$ which is finite. In particular, ${\rm Ker}\,g$ is
finite dimensional and hence lies in $\rep^-(Q)$. Since $\rep^-(Q)$
has cokernels, $M/M_{\Omega} \in \rep^-(Q)$. The fact that
$M_{\Omega} \in \rep^+(Q)$ follows from the fact that $\rep^+(Q)$ is
closed under extensions. Hence, the bottom row of the above diagram
yields that $M$ is an extension-representation of $M/M_{\Omega} \in \rep^-(Q)$ by
$M_{\Omega} \in \rep^+(Q)$ which is finite by Lemma
\ref{FiniteExtLemma}.
\end{proof}

\section{Properties of $\rrep(Q)$}

The purpose of this section is to collect some properties of
$\rrep(Q)$ and its objects. Using the fact that both $\rep^+(Q)$ and
$\rep^-(Q)$ are abelian and Hom-finite, we can prove that $\rrep(Q)$
is abelian and Hom-finite and consequently is \emph{Krull-Schmidt}, that is, the Krull-Remak-Schmidt theorem holds: every object decomposes uniquely (up to isomorphism and permutation) as a finite direct sum of objects having local endomorphism algebras. Let us first prove that it
is abelian.

\begin{Prop}
The category $\rrep(Q)$ is abelian.
\end{Prop}

\begin{proof}
Let us first prove that $\rrep(Q)$ has kernels.  Let $f: M_1 \to
M_2$ be a morphism with $M_1,M_2 \in \rrep(Q)$.  It is easy to see
that $M_1 \oplus M_2 \in \rrep(Q)$. By Lemma \ref{StandardExt},
there exists a full subquiver $\Omega$ of $Q$ such that $(M_1 \oplus
M_2)_{\Omega} \in \rep^+(Q)$ and $(M_1 \oplus M_2) / (M_1 \oplus
M_2)_{\Omega} \in \rep^-(Q)$. Now, $(M_i)_{\Omega} \in \rep^+(Q)$,
for $i=1,2$, since $\rep^+(Q)$ is closed under direct summand.
Similarly, $M_i/(M_i)_{\Omega} \in \rep^-(Q)$, for $i=1,2$. By Lemma
\ref{FiniteExtLemma}, we have finite extensions
$$0 \to (M_i)_{\Omega} \to M_i \to M_i/(M_i)_{\Omega} \to 0$$
for $i=1,2$. Thus, one has a diagram
$$\xymatrix{0 \ar[r] & (M_1)_{\Omega} \ar[r] \ar[d]^u & M_1 \ar[r] \ar[d]^f & M_1/(M_1)_{\Omega} \ar[r] \ar[d]^h & 0 \\
0 \ar[r] & (M_2)_{\Omega} \ar[r] & M_2 \ar[r] & M_2/(M_2)_{\Omega}
\ar[r] & 0}$$ where $u$ is the restriction of $f$ to $\Omega$ and
$h$ is the induced morphism.  This yields an exact sequence
$$0 \to {\rm Ker}\,u \to {\rm Ker}\,f \stackrel{v'}{\rightarrow} {\rm Ker}\,h \stackrel{v}{\rightarrow} {\rm Coker}\,u$$
in $\rep(Q)$. Now, ${\rm Ker}\,u, {\rm Coker}\,u \in \rep^+(Q)$
while ${\rm Ker}\,h \in \rep^-(Q)$.  Since $${\rm supp}({\rm Im}\,v)
\subseteq {\rm supp}({\rm Coker}\,u) \cap {\rm supp}({\rm
Ker}\,h),$$ ${\rm Im}\,v$ is finite dimensional. Hence, ${\rm
Ker}\,v \in \rep^-(Q)$ since ${\rm Ker}\,h/{\rm Ker}\,v$ is finite
dimensional. Therefore, ${\rm Im}\,v' \in \rep^-(Q)$, meaning that
${\rm Ker}\,f$ is an extension-representation of ${\rm Im}\,v' \in \rep^-(Q)$ by
${\rm Ker}\,u \in \rep^+(Q)$. Now, if ${\rm Ker}\,f$ is not a finite
extension-representation of ${\rm Im}\,v'$ by ${\rm Ker}\,u$, it means that there
exists an infinite family of arrows $\{\alpha_i\}_{i\ge 1}$ such
that for $i \ge 1$, $\alpha_i$ starts at a vertex in ${\rm
supp}({\rm Im}\,v')$, ends at a vertex in ${\rm supp}({\rm Ker}\,u)$
and $({\rm Ker}\,f)(\alpha_i)$ is non-zero. Observe that $${\rm
supp}({\rm Im}\,v') \subseteq {\rm supp}(M_1/(M_1)_{\Omega})\;\;
\text{and}\;\;{\rm supp}({\rm Ker}\,u) \subseteq {\rm
supp}((M_1)_{\Omega}).$$ Moreover, $({\rm Ker}\,f)(\alpha_i) \ne 0$
implies that $M_1(\alpha_i) \ne 0$. Hence, for $i \ge 1$, $\alpha_i$
is an arrow starting at a vertex in ${\rm supp}(M_1/(M_1)_{\Omega})$
and ending at a vertex in ${\rm supp}((M_1)_{\Omega})$ such that
$M_1(\alpha_i) \ne 0$.  This contradicts the fact that the given
extension of $M_1/(M_1)_{\Omega}$ by $(M_1)_{\Omega}$ is finite.
This shows that ${\rm Ker}\,f \in \rrep(Q)$.

Similarly, one can prove that $\rrep(Q)$ has cokernels.  Hence,
being a full subcategory of $\rep(Q)$, $\rrep(Q)$ is abelian.
\end{proof}

\begin{Prop}
The category $\rrep(Q)$ is Hom-finite and Krull-Schmidt.
\end{Prop}

\begin{proof}
Recall that both $\rep^+(Q)$ and $\rep^-(Q)$ are Hom-finite; see
\cite{BLP}. Moreover, if $M \in \rep^+(Q)$ and $N \in \rep^-(Q)$,
then $\Hom(M,N)$ and $\Hom(N,M)$ are finite dimensional since ${\rm
supp}(M) \cap {\rm supp}(N)$ is finite. Let $M_1,M_2 \in \rrep(Q)$
such that $M_i$ is a finite extension-representation of $N_i$ by $L_i$, with $L_i
\in \rep^+(Q)$ and $N_i \in \rep^-(Q)$, $i=1,2$. Then one has
$$0 \to \Hom(N_1,L_2) \to \Hom(M_1,L_2) \to \Hom(L_1,L_2)$$
showing that $\Hom(M_1,L_2)$ is finite dimensional.  Similarly,
$\Hom(M_1,N_2)$ is finite dimensional.  The exact sequence
$$0 \to \Hom(M_1,L_2) \to \Hom(M_1,M_2) \to \Hom(M_1,N_2)$$
yields that $\Hom(M_1,M_2)$ is finite dimensional.  Now, it is well known that a Hom-finite abelian category is Krull-Schmidt. \qedhere
\end{proof}

\begin{Prop} \label{closedext}
The category $\rrep(Q)$ is closed under finite extensions and direct
summands.
\end{Prop}

\begin{proof}
Let
$$(*): 0 \to M_1 \to M_3 \to M_2 \to 0$$
be a finite extension in $\rep(Q)$ with $M_1,M_2 \in \rrep(Q)$. By
Lemma \ref{StandardExt}, there exist top-finite successor-closed
subquivers $\Omega', \Omega''$ of $Q$ such that
$(M_1)_{\Omega'},(M_2)_{\Omega''} \in \rep^+(Q)$ and
$M_1/(M_1)_{\Omega'},M_2/(M_2)_{\Omega''} \in \rep^-(Q)$. Let
$\Omega$ be the union of $\Omega'$ and $\Omega''$, which is
top-finite and successor-closed in $Q$. Consider the following
commutative diagram with exact rows:
$$\xymatrixrowsep{12 pt}\xymatrixrowsep{15 pt}\xymatrix{0 \ar[r] & (M_1)_{\Omega'} \ar[r] \ar[d] & M_1 \ar[r] \ar@{=}[d] & M_1/(M_1)_{\Omega'} \ar[r] \ar[d] & 0 \\
0 \ar[r] & (M_1)_{\Omega} \ar[r] & M_1 \ar[r] & M_1/(M_1)_{\Omega}
\ar[r] & 0.}$$ Since the inclusion $(M_1)_{\Omega'} \to
(M_1)_{\Omega}$ has a cokernel which is supported by the
intersection of ${\rm supp}(M_1 / (M_1)_{\Omega'})$ with $\Omega''$,
which is a finite quiver, we see that $(M_1)_{\Omega} \in
\rep^+(Q)$.   Also, the kernel of the projection
$M_1/(M_1)_{\Omega'} \to M_1/(M_1)_{\Omega}$ is finite dimensional
showing that $M_1/(M_1)_{\Omega}$ is finitely co-presented.
Similarly, $(M_2)_{\Omega} \in \rep^+(Q)$ and $M_2/(M_2)_{\Omega}
\in \rep^-(Q)$. By Lemma \ref{FiniteExtLemma}, the extensions
$$0 \to (M_i)_{\Omega} \to M_i \to M_i/(M_i)_{\Omega} \to 0$$
are finite for $i=1,2$.  By restricting the extension $(*)$ to
$\Omega$ and $Q \backslash \Omega$, one gets a commutative diagram
$$\xymatrixrowsep{12 pt}\xymatrixrowsep{15 pt}\xymatrix{0 \ar[r] & (M_1)_{\Omega} \ar[r] \ar[d] & (M_3)_{\Omega} \ar[r] \ar[d] & (M_2)_{\Omega} \ar[r] \ar[d] & 0 \\
0 \ar[r] & M_1 \ar[r]\ar[d] & M_3 \ar[r]\ar[d] & M_2 \ar[r]\ar[d] &
0\\
0 \ar[r] & M_1/(M_1)_{\Omega} \ar[r] & M_3/(M_3)_{\Omega} \ar[r] &
M_2/(M_2)_{\Omega} \ar[r] & 0}$$ where all rows and all columns are
exact. Since $\rep^+(Q)$ and $\rep^-(Q)$ are closed under
extensions, $(M_3)_{\Omega} \in \rep^+(Q)$ and $M_3/(M_3)_{\Omega}
\in \rep^-(Q)$.  If $M_3$ is not a finite extension-representation
of $M_3/(M_3)_{\Omega}$ by $(M_3)_{\Omega}$, then there exists an
infinite family of arrows $\{\alpha_i\}_{\i \ge 1}$ all starting at
a vertex in ${\rm supp}(M_3/(M_3)_{\Omega}) \subseteq Q \backslash
\Omega$ and ending at a vertex in ${\rm supp}((M_3)_{\Omega})
\subseteq \Omega$ such that $M_3(\alpha_i) \ne 0$. Since the middle
row is a finite extension, $(M_1\oplus M_2)(\alpha_i)\ne 0$ for all
but a finite number of $i \ge 1$, which yields that $M_1(\alpha_i)
\ne 0$ for all but a finite number of $i \ge 1$ or $M_2(\alpha_i)
\ne 0$ for all but a finite number of $i \ge 1$. In the first case,
we get that the first column is not a finite extension, and in the
second case, that the third column is not a finite extension.  This
is a contradiction. This shows that $M_3 \in \rrep(Q)$. The last
part of the proposition follows from the fact that $\rrep(Q)$ is
abelian.
\end{proof}


\smallskip

In order to have a better understanding of the objects in
$\rrep(Q)$, we first state the following result, which can be
derived easily from \cite[Theorem 1.12]{BLP}.

\begin{Prop}[Bautista-Liu-Paquette] \label{PropBLP} Let $0 \ne M \in \rep^+(Q)$ have support
$\Sigma$. Then there exists a co-finite and successor-closed
subquiver $\Omega$ of $\Sigma$ such that
\begin{enumerate}[$(1)$]
    \item $M_{\Omega}$ is projective,
    \item $M_{\Sigma\backslash \Omega}$ is non-zero, finite dimensional
    and is
indecomposable when $M$ is.
\end{enumerate}
Moreover, any co-finite and successor-closed subquiver $\Omega'$ of
$\Omega$ also satisfies properties $(1)$ and $(2)$.
\end{Prop}

Of course, the dual result for $\rep^-(Q)$ holds true. We will show
that a similar statement in $\rrep(Q)$ can be obtained. Before going
further, we need a lemma. We say that a representation $M \in
\rep(Q)$ is \emph{indecomposable up to projectives} if $M = M_1
\oplus M_2$ implies that $M_1$ or $M_2$ is projective in
$\rep^+(Q)$.  We also have the dual notion of \emph{indecomposable
up to injectives}. Finally, $M \in \rep(Q)$ is \emph{indecomposable
up to projectives and injectives} if $M = M_1 \oplus M_2$ implies
that one of $M_1,M_2$ is projective in $\rep^+(Q)$ or injective in
$\rep^-(Q)$.

\begin{Lemma} \label{ExtProjFinCoPres}
Let $M \in \rrep(Q)$ have support $\Sigma$.  Then there exists a
top-finite and successor-closed subquiver $\Omega$ of $\Sigma$ such
that
\begin{enumerate}[$(1)$]
    \item $M_{\Omega} \in \rep^+(Q)$ is projective,
    \item $M/M_{\Omega} \in \rep^-(Q)$,
    \item If $M/M_{\Omega}$ is indecomposable, then $M$
is indecomposable up to projectives,
    \item If $M$ is indecomposable,
then $M/M_{\Omega}$ is indecomposable.
\end{enumerate}
Moreover, any successor-closed and co-finite subquiver of $\Omega$
also satisfies properties $(1)$ to $(4)$.
\end{Lemma}

\begin{proof}
By Lemma \ref{StandardExt}, there exists a top-finite full subquiver
$\Omega'$ of $\Sigma$ such that $M_{\Omega'} \in \rep^+(Q)$ and
$M/M_{\Omega'} \in \rep^-(Q)$. We may choose $\Omega'$ to be
successor-closed in $\Sigma$. If $\Omega'$ is empty, then we are
done.  Suppose that $\Omega'$ is non-empty. By Proposition
\ref{PropBLP}, let $\Sigma_P$ be a co-finite and successor-closed
subquiver of $\Omega'$ such that $(M_{\Omega'})_{\Sigma_P} =
M_{\Sigma_P}$ is projective in $\rep^+(Q)$ and $M_{\Omega'} /
M_{\Sigma_P}$ is finite dimensional, non-zero and indecomposable
whenever $M_{\Omega'}$ is. Clearly, $M_{\Sigma_P} \in \rep^+(Q)$ and
$M/M_{\Sigma_P} \in \rep^-(Q)$.
Suppose that $\Omega$ is a co-finite and successor-closed subquiver
of $\Sigma_P$ such that if $\alpha: x\to y$ is an arrow with $y \in
\Omega$ and $M(\alpha) \ne 0$, then $x \in \Sigma_P$. Such a
co-finite subquiver of $\Sigma_P$ exists since $M$ is a finite
extension-representation of $M/M_{\Sigma_P}$ by $M_{\Sigma_P}$ by Lemma
\ref{StandardExt}. Moreover, one can chose $\Omega$ so that
$\Sigma_P \backslash \Omega$ contains the support of the top of the
projective representation $M_{\Sigma_P}$. Being a sub-representation
of $M_{\Sigma_P}$, $M_{\Omega}$ is projective.  Moreover, it lies in
$\rep^+(Q)$ since $\Omega$ is co-finite in $\Sigma_P$. Similarly,
$M/M_{\Omega} \in \rep^-(Q)$.

Now, suppose that $M$ decomposes non-trivially as $M=M_1 \oplus M_2$
with $M_1,M_2$ non-projective. If the support of $M_1$ is included
in $\Omega$, then $M_1$ is a sub-representation of $M_{\Omega}$ and
hence is projective, a contradiction. Thus, the support of $M_1$ has
an intersection with $\Sigma \backslash \Omega$, meaning that the
restriction $M_1 /(M_1)_{\Omega}$ of $M_1$ to $\Sigma \backslash
\Omega$ yields a non-zero direct summand of $M/M_{\Omega}$. If
$M_1/(M_1)_{\Omega} = M/M_{\Omega}$, then $M_2$ has a support
included in $\Omega$ and hence is projective, a contradiction. This
shows (3).

Conversely, let $N = M/M_{\Omega} = M_{\Sigma \backslash \Omega}$
with a non-trivial decomposition $N=N_1 \oplus N_2$. Let $M'$ be a
the sub-representation of $M$ generated by the elements $N_1(x)
\subseteq N(x) = M(x)$ for $x$ in the support of $N$. Let $i_x:
M'(x) \to M(x)$, for $x \in Q_0$, be the inclusions defining the
inclusion morphism $M' \to M$. Observe now that $(M')_{\Sigma
\backslash \Omega} = N_1$. For $x \in Q \backslash \Omega$, there
are maps $t_x: M(x)=N(x) \to M'(x)=N_1(x)$ such that
$t_xi_x=1_{M'(x)}$ and for any arrow $\alpha : x \to y$ with $x,y
\in Q \backslash \Omega$, $M'(\alpha)t_x = t_yM(\alpha)$. Observe
that $(M')_{ \Sigma_P}, M_{\Sigma_P}$ are projective representations
whose tops are supported by $\Sigma_P \backslash \Omega$. Hence, the
maps $t_x$, $x \in \Sigma_P \backslash \Omega$, provide an
epimorphism $M_{\Sigma_P \backslash \Omega} \to (M')_{\Sigma_P
\backslash \Omega }$ which could be extended to an epimorphism $t:
M_{\Sigma_P} \to (M')_{\Sigma_P}$ between projective
representations. Therefore, $t$ is a retraction. Thus, for $x \in
Q$, we have maps $t_x : M(x) \to M'(x)$ which are compatible with
the arrows in $Q \backslash \Omega$ and in $\Omega$.  Since any
other arrow of $Q_1$ is not supporting $M$, the $t_x$ define an
epimorphism $M \to M'$ and is such that $t_xi_x = 1_{M'(x)}$ for $x
\in Q \backslash \Omega$. It only remains to show that
$t_xi_x=1_{M'(x)}$ for $x \in \Omega$. Recall that $\Omega$ is
top-finite and hence, for $x \in \Omega$, there is a non-negative
integer $n_x$ such that every path $y \rightsquigarrow x$ with $y
\in \Omega$ has length bounded by $n_x$. We proceed by induction on
$n_x$.  If $n_x=0$ and every arrow $\alpha: y \to x$ in $Q$ is not
supporting $M$, then the elements in $M(x)$ are top elements of
$M_{\Sigma_P}$, contradicting the fact the $\Omega$ does not contain
the vertices supporting the top of $M_{\Sigma_P}$. Let
$\{\alpha_i:y_i \to x\}_{1\le i \le r}$ be the non-empty set of
arrows ending in $x$ and supporting $M$. Since $(M')_{\Sigma_P}$ is
a projective representation whose top is supported by $\Sigma_P
\backslash \Omega$ and $y_i \in \Sigma_P$ for $1 \le i \le r$, the
map
$$(M'(y_i)): \oplus_{i=1}^rM'(y_i) \to M'(x)$$
is bijective.  Moreover, $t_{y_i}i_{y_i}=1_{M'(y_i)}$ for $1 \le i
\le r$.  Now, \begin{eqnarray*}t_xi_x(M'(y_i)) & = &
t_x(M(y_i))(i_{y_i})\\ & = & (M'(y_i))(t_{y_i})(i_{y_i})
\\ & = &(M'(y_i))\end{eqnarray*} showing that $t_xi_x=1_{M'(x)}$.
Now, if $n_x > 0$, then every arrow $\alpha : y \to x$ supporting
$M$ is such that $y \in \Sigma_P$ or $y \in \Omega $ with $n_y <
n_x$.  Hence, by induction, $t_yi_y = 1_{M'(y)}$.  The proof then
uses the same argument as above to show that $t_xi_x=1_{M'(x)}$.
This shows that $M'$ is a direct summand of $M$, which is
non-trivial and proper since $N_1$ is a non-trivial proper direct
summand of $N$. Hence, $M$ is decomposable, showing (4). The last
part of the lemma is easy to see. \qedhere
\end{proof}


The following result says how the representations in $\rrep(Q)$ are
constructed.

\begin{Prop} \label{PropProjFiniteInj}
Let $M \in \rrep(Q)$ with support $\Sigma$.  There exist full
subquivers $\Sigma_P$ and $\Sigma_I$ of $\Sigma$ such that
\begin{enumerate}[$(1)$]
    \item $\Sigma_P$ is top-finite and successor-closed in $\Sigma$ such that $M_{\Sigma_P}$ is
    projective,
    \item $\Sigma_I$ is socle-finite and predecessor-closed in $\Sigma \backslash \Sigma_P$ such that $M_{\Sigma_I}$ is
    injective,
    \item $\Omega:=\Sigma \backslash (\Sigma_P \cup \Sigma_I)$ is
finite and non-empty, 
 \item $M_{\Omega}$ is indecomposable whenever $M$ is indecomposable; and $M$ is indecomposable up to projectives and injectives whenever $M_{\Omega}$ is indecomposable.
\end{enumerate}
Moreover, if $\Sigma_P'$ is co-finite and successor-closed in
$\Sigma_P$ and $\Sigma_I'$ is co-finite and predecessor-closed in
$\Sigma_I$, then $\Sigma_I'$ and $\Sigma_P'$ also satisfies
properties $(1)$ to $(4)$.
\end{Prop}

\begin{proof}
By Lemma \ref{ExtProjFinCoPres}, there exists a successor-closed and
top-finite subquiver $\Sigma_P$ of $\Sigma$ such that $M_{\Sigma_P}$
is projective in $\rep^+(Q)$ and $M/M_{\Sigma_P} \in \rep^-(Q)$.
Moreover, if $M/M_{\Sigma_P}$ is indecomposable, then $M$ is
indecomposable up to projectives; and if $M$ is indecomposable, then
$M/M_{\Sigma_P}$ is.  Since $M/M_{\Sigma_P} \in \rep^-(Q)$, by the
dual of Proposition \ref{PropBLP}, there exists a predecessor-closed
subquiver $\Sigma_I$ of $\Sigma \backslash \Sigma_P$ such that
$\Sigma_I$ is co-finite in $\Sigma \backslash \Sigma_P$ and
$(M/M_{\Sigma_P})_{\Sigma_I} = M_{\Sigma_I}$ is injective in
$\rep^-(Q)$. It is easy to see that $\Sigma \backslash (\Sigma_P
\cup \Sigma_I)$ is finite. By Proposition \ref{PropBLP}, $\Sigma_P$
and $\Sigma_I$ can be chosen so that $\Omega:=\Sigma \backslash
(\Sigma_P \cup \Sigma_I)$ is non-empty.  Moreover, since $M \in
\rrep(Q)$, we can assume that $\Omega$ is large enough so that any
arrow attached to $\Sigma_P$ and $\Sigma_I$ does not support $M$.
Suppose now that $M$ is indecomposable. Then $M/M_{\Sigma_P}$ is
indecomposable by Lemma \ref{ExtProjFinCoPres}. Moreover, by the
dual of Proposition \ref{PropBLP}, one can choose $\Sigma_I$ in such
a way that
$$(M/M_{\Sigma_P})_{\Sigma \backslash \Sigma_I} = M_{\Omega}$$ is indecomposable.
Conversely, suppose that $M_{\Omega}$ is indecomposable. Then any
non-trivial decomposition of $M$ yields an indecomposable direct
summand $Z$ of $M$ supported by $\Sigma_I
\cup \Sigma_P$. 
Since no arrow supporting $Z$ joins $\Sigma_P$ to $\Sigma_I$, we see
that ${\rm supp}(Z) \subseteq \Sigma_P$ or ${\rm supp}(Z) \subseteq
\Sigma_I$. In the first case, $Z$ is a direct summand of
$M_{\Sigma_P}$ and hence is projective in $\rep^+(Q)$. In the second
case, $Z$ is injective in $\rep^-(Q)$. The last part of the
statement follows similarly.
\end{proof}

\section{Irreducible morphisms in $\rrep(Q)$}

Let $\C$ be any additive $k$-category.  A morphism $f:X\to Y$ is
said to be \emph{irreducible} if it is neither a section nor a
retraction, and any factorization $f = gh$ imply that $h$ is a
section or $g$ is a retraction. In this section, we prove that the
irreducible morphisms in $\rrep(Q)$ are all contained in the
Auslander-Reiten sequences of $\rep(Q)$.

\smallskip

For simplicity, an indecomposable representation $M \in \rrep(Q)$
which is neither finitely presented nor finitely co-presented will
be called \emph{doubly-infinite}, since its support contains a
left-infinite path and a right-infinite path.

\smallskip

Now, we need some notations for the lemmas presented in this
section. Fix $M,N$ two doubly-infinite indecomposable
representations in $\rrep(Q)$ with a non-isomorphism $f : M \to N$.
Let $\Sigma$ be the support of $M \oplus N$. We can deduce from
Proposition \ref{PropProjFiniteInj} that there exist a
successor-closed subquiver $\Sigma_P$ of $\Sigma$ and a
predecessor-closed subquiver $\Sigma_I$ of $\Sigma \backslash
\Sigma_P$ such that $M_{\Sigma_P},N_{\Sigma_P}$ are projective in
$\rep^+(Q)$, $M_{\Sigma_I}, N_{\Sigma_I}$ are injective in
$\rep^-(Q)$ and $M_{\Theta},N_{\Theta}$, where $\Theta = \Sigma
\backslash (\Sigma_P \cup \Sigma_I)$, are finite dimensional and
indecomposable.

Since $\Sigma_I$ is clearly infinite, there exists a finite
successor-closed subquiver $\Theta'$ of $\Sigma_I$ such that:

\begin{enumerate}[$(1)$]
\item $\Theta'$ supports the socle of $(M\oplus N)_{\Sigma_I}$,
    \item every arrow $x \to y$ supporting $M \oplus N$ with $x \in
\Sigma_I \backslash \Theta'$ is such that $y \in \Sigma_I$.
\end{enumerate}
\smallskip
\noindent Such a finite quiver $\Theta'$ exists since $M\oplus N$ is
a finite extension-representation of $(M \oplus N)_{\Sigma_I}$ by
$(M \oplus N)_{\Sigma \backslash \Sigma_I}$ by Lemma
\ref{FiniteExtLemma}. Now, there exist a vertex $a \in \Sigma_I
\backslash \Theta'$ and an arrow $a \to b$ with $b \in \Theta'$.
Then $a$ does not lie in the support of ${\rm soc}(M \oplus
N)_{\Sigma_I}$ and we can choose it so that every arrow starting in
$a$ and supporting $M \oplus N$ has an ending point in $\Theta'$.
The successor-closed subquiver of $\Sigma_I$ generated by $\Theta'$
and $a$ will be denoted by $\Lambda$. Observe that $\Lambda$ is
finite and successor-closed in $\Sigma_I$.
Now, set $\Delta : = (\Sigma \backslash \Sigma_I) \cup \Lambda$ and
$\Delta' := \Theta \cup \Lambda = \Delta \backslash \Sigma_P$; see
figure $1$.

\begin{figure}[t]
\begin{center}
\includegraphics[width=0.38\textwidth]{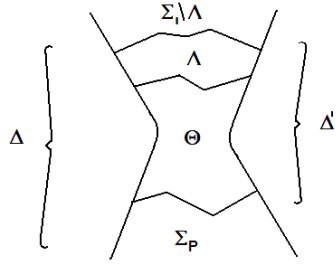}
\end{center}
\caption{The subdivisions of the quiver $\Sigma$}

\end{figure}

\smallskip

First, let us show that $f_{\Sigma \backslash \Sigma_I}$ can be
assumed not be an isomorphism. Otherwise, there exists a co-finite
and predecessor-closed subquiver $\Sigma'_I$ of $\Sigma_I$ such that
$f_{\Sigma \backslash \Sigma'_I}$ is not an isomorphism. By the
second part of Proposition \ref{PropProjFiniteInj}, $\Sigma_P$ and
$\Sigma'_I$ satisfy the same properties stated in the introduction
of this section.  We then set $\Sigma_I :=\Sigma'_I$. The subquivers
$\Theta$ and $\Lambda$ need also to be changed according to the new
definition of $\Sigma_I$. We start with the following lemma.

\begin{Lemma} \label{psi}
Let $L \in \rep(Q)$ such that ${\rm supp}(L) \subseteq \Sigma$ and
any arrow $\alpha : x \to y$ supporting $L$ with $x \in \Sigma_I
\backslash \Lambda$ is such that $y \in \Sigma_I$. Then the
restriction map:
$$\psi_L: \Hom(L,X) \to \Hom(L_{\Delta},X_{\Delta})$$
is an isomorphism of $k$-vector spaces, where $X=M$ or $X=N$.
\end{Lemma}

\begin{proof} We only consider the case where $X=N$. It is clear that $\psi_L$ is $k$-linear. If $h : L \to N$ is such that $h_{\Delta}=0$,
then $h_{\Lambda}=0$.  Now, consider the morphism $h_{\Sigma_I}:
L_{\Sigma_I} \to N_{\Sigma_I}$ where $N_{\Sigma_I}$ is injective in
$\rep(Q)$. By the construction of $\Lambda$, $N_{\Lambda}$ is an
essential sub-representation of $N_{\Sigma_I}$. Therefore, if ${\rm
Im}(h_{\Sigma_I})$ is non-zero, then it has a non-zero intersection
with $N_{\Lambda}$, contradicting the fact that $h_{\Delta}=0$.
Hence, $h_{\Sigma_I}=0$ yielding $h=0$. This shows that $\psi_L$ is
injective. Conversely, let $g: L_{\Delta} \to N_{\Delta}$ be any
morphism. Consider the canonical inclusions $i_L: L_{\Lambda} \to
L_{\Sigma_I}$ and $i_N : N_{\Lambda} \to N_{\Sigma_I}$. Since
$N_{\Sigma_I}$ is injective, there exists a morphism $v:L_{\Sigma_I}
\to N_{\Sigma_I}$ such that $vi_L=i_Ng_{\Lambda}$. Now, the
morphisms $g : L_{\Delta} \to N_{\Delta}$ and $v:L_{\Sigma_I} \to
N_{\Sigma_I}$ coincide on $\Lambda$, that is, $g_{\Lambda} =
v_{\Lambda}$. Since there is no arrow $x \to y$ with $x \in \Sigma_I
\backslash \Lambda$ and $y \in \Sigma\backslash \Sigma_I$ which
supports $L \oplus N$, we see that $g$ and $v$ yield a morphism $h:
L \to N$ such that $h_{\Delta} = g$. This shows that $\psi_L$ is
surjective and thus that it is an isomorphism.
\end{proof}

\begin{Lemma} \label{indecDelta}
Both $M_{\Delta}$ and $N_{\Delta}$ are indecomposable.
\end{Lemma}

\begin{proof}
By Proposition \ref{PropProjFiniteInj}, $M_{\Delta'}$ is
indecomposable.  Therefore, if $M_{\Delta}$ decomposes
non-trivially, then there is a non-zero direct summand $Z$ of
$M_{\Delta}$ supported by $\Sigma_P$.  Since there is no arrow from
$\Sigma_I\backslash \Lambda$ to $\Sigma_P$ supporting $M$, we see
that $Z$ is a direct summand of $M$, a contradiction.  Thus,
$M_{\Delta}$ is indecomposable and similarly, $N_{\Delta}$ is
indecomposable.
\end{proof}

\begin{Lemma} \label{sublemma1} If $f$ is an irreducible monomorphism, then $f_{\Delta}$ is irreducible in
$\rep^+(\Delta)$.
\end{Lemma}

\begin{proof} Suppose that $f: M \to N$ is an irreducible monomorphism, which will be assumed to be an inclusion. By Lemma \ref{psi},
$f_{\Delta}$ is a monomorphism which is neither a section nor a
retraction. Suppose that $f_{\Delta}=vu$ where $u : M_{\Delta} \to
L$, $v: L \to N_{\Delta}$ and $L \in \rep^+(\Delta)$. Since
$\Lambda$ is successor-closed in $\Sigma_I$, it is a convex full
subquiver of $\Sigma_I$.  Hence, $M_{\Lambda}$ and $N_{\Lambda}$ are
injective in $\rep(\Lambda)$ since $M_{\Sigma_I}$ and $N_{\Sigma_I}$
are injective in $\rep(Q)$. Therefore, $u_{\Lambda}$ and
$f_{\Lambda}$ are section maps. Hence, $L_{\Lambda} =
M_{\Lambda}\oplus Z$ where $Z \in \rep(\Lambda)$. Thus, one may
choose to write $u_{\Lambda} = (\id, 0)^T$ and $v_{\Lambda} = (s,
g)$ where $s$ is a section. Since $N_{\Sigma_I}$ is injective, there
exists $s' : M_{\Sigma_I} \to N_{\Sigma_I}$ such that
$(s')_{\Lambda} = s$. Let $i: N_{\Lambda} \to N_{\Sigma_I}$ be the
canonical inclusion. Consider the representation $L' = M_{\Sigma_I}
\oplus Z$ of $\Sigma_I$ with $u' : M_{\Sigma_I} \to L'$ and $v' : L'
\to N_{\Sigma_I}$ where $u' = (\id, 0)^T$ and $v' = (s',ig)$. It is
clear that $(v'u')_{\Lambda} = vu$. Now, the representation $L'$ of
$\Sigma_I$ together with the representation $L$ of $\Delta$ are such
that $(L')_{\Lambda} = L_{\Lambda}$ and thus yield a representation
$L''$ of $\Sigma$ such that if $\alpha: x \to y$ is an arrow with
one extremity in $\Sigma_I \backslash \Lambda$ and the other in
$\Sigma \backslash \Sigma_I$, then $L''(\alpha)=0$. It is easy to
see that $L'' \in \rrep(Q)$. Observe also that if $\alpha: x \to y$
is an arrow with one extremity in $\Sigma_I \backslash \Lambda$ and
the other in $\Sigma \backslash \Sigma_I$, then $M(\alpha)=0$. Using
this and the fact that $\Delta \cap \Sigma_I = \Lambda$, the
morphisms $u: M_{\Delta} \to L$ and $u' : M_{\Sigma_I} \to L'$ give
a morphism $u'' : M \to L''$ such that $(u'')_{\Delta} = u$ and
$(u'')_{\Sigma_I}=u'$. Similarly, the morphisms $v:  L \to
N_{\Delta}$ and $v' : L' \to N_{\Sigma_I}$ give a morphism $v'' :
L'' \to N$ such that $(v'')_{\Delta} = v$ and $(v'')_{\Sigma_I}=v'$.
Hence, $f = v''u''$ by Lemma \ref{psi}. Since $f$ is irreducible in
$\rrep(Q)$, either $u''$ is a section, or $v''$ is a retraction.
Thus, $u$ is a section or $v$ is a retraction.  This shows that
$f_{\Delta}$ is irreducible.
\end{proof}

We are now ready to prove the main result of this section.

\begin{Prop} \label{irred}
Let $f : M \to N$ be an irreducible morphism in $\rrep(Q)$ with
$M,N$ indecomposable.  Then we have four possible cases.
\begin{enumerate}[$(1)$]
    \item $M,N \in \rep^+(Q)$,
    \item $M,N \in \rep^-(Q)$,
    \item $M \in \rep^-(Q)$ is infinite dimensional and $N$ is doubly-infinite,
    \item $N \in \rep^+(Q)$ is infinite dimensional and $M$ is doubly-infinite.
\end{enumerate}
\end{Prop}

\begin{proof}
Suppose that $N \in \rep^+(Q)$.  If $N$ is projective, then the
inclusion ${\rm rad}(N) \to N$ is minimal right almost split in
$\rep(Q)$ and hence also minimal right almost split in $\rrep(Q)$.
Therefore, $M$ is a direct summand of the radical of $N$ and hence
is projective in $\rep^+(Q)$. If $N$ is not projective, then one has
an almost split sequence
$$0 \to N' \to E \to N \to 0$$
in $\rep(Q)$ with $N' \in \rep^-(Q)$, which is also an almost split
sequence in $\rrep(Q)$ by Lemma \ref{Ext+-}. Then $M$ is a direct
summand of $E$. If $N'$ is finite dimensional, then $M \in
\rep^+(Q)$. If $N$ is finite dimensional, then $M \in \rep^-(Q)$.
Otherwise, by \cite[Corolarry 3.3]{BLP}, either $M$ is finite
dimensional or doubly-infinite. Hence, if $N \in \rep^+(Q)$, then
(1), (2) or (4) hold. Dually, if $M \in \rep^-(Q)$, (1), (2) or (3)
hold.

We need to show that $M,N$ cannot be both doubly-infinite.  Suppose
it is the case.  We only consider the case where $f$ is a
monomorphism.  The case where $f$ is an epimorphism is treated in a
similar way.

We may assume that $f$ is an inclusion, that is, $M$ is a
sub-representation of $N$. We set $\Sigma$ to be the support of $M
\oplus N$ and use the notation introduced at the beginning of this
section for the subquivers $\Sigma_P, \Sigma_I, \Delta, \Delta',
\Theta$ and $\Lambda$ of $\Sigma$. By Lemma \ref{indecDelta},
$M_{\Delta},N_{\Delta}$ are indecomposable. By Lemma
\ref{sublemma1}, $f_{\Delta}$ is an irreducible monomorphism.
Let us first assume that $N_{\Delta}$ is projective in
$\rep^+(\Delta)$. Then $M_{\Delta}$ is also projective in
$\rep^+(\Delta)$ and $M_{\Delta}$ is a direct summand of the radical
of $N_{\Delta}$. Since $a$ is a source vertex in $\Delta$,
$N_{\Delta} \cong P_a$ and hence $a$ is not in the support of
$M_{\Delta}$. Now, there exists an arrow $a \to c$ in $\Delta$ with
$c \in {\rm supp}(M_{\Delta})$.  By the construction of $\Lambda$,
$c \in \Lambda$. Since $M_{\Sigma_I}$ is injective in $\rep(Q)$,
$M_{\Lambda}$ is injective in $\rep(\Lambda)$. Therefore,
$f_{\Lambda}$ is a section.  Since $N_{\Lambda}$ has a simple top,
it is indecomposable.  Hence $f_{\Lambda}$ is an isomorphism. But
this is impossible since $M_{\Lambda}(a)=0$. This contradiction
shows that $N_{\Delta}$ is not projective. Therefore, we have an
almost split sequence
$$\eta: \quad 0 \to W' \to E' \to N_{\Delta} \to 0$$ in $\rep(\Delta)$.

Observe that $M_{\Delta}$ is a direct summand of $E'$ and $W'$ is
finite dimensional since $W' \in \rep^-(\Delta) = \rep^b(\Delta)$.
Thus, we get an irreducible monomorphism $W' \to M_{\Delta}$ whose
image $W$ is a proper sub-representation of $M_\Delta$. We claim
that $W_{\Lambda} \ne M_{\Lambda}$. Suppose the contrary. In
particular, $W_{\Lambda}$ is injective in $\rep(\Lambda)$, and also
in $\rep^+(\Delta)$ since $\Lambda$ is predecessor-closed in
$\Delta$. By the dual of \cite[Lemma 2.5]{BLP}, $N_{\Delta}$ is
constructed from $W \cong W'$ in the following way. Take a minimal
injective co-resolution
$$(*): \quad0 \to W \to \textstyle{\bigoplus}_{i=1}^rI'_{x_i} \stackrel{}{\rightarrow} \textstyle{\bigoplus}_{j=1}^sI'_{y_j} \to 0$$
of $W$ where the $x_i,y_j$ are vertices in $\Delta$ and $I'_x$, for
$x \in \Delta_0$, denotes the injective representation at $x$ in
$\rep^-(\Delta)$. Then $N_{\Delta}$ is (isomorphic to) the cokernel
of the corresponding map
$$\textstyle{\bigoplus}_{i=1}^rP'_{x_i} \to
\textstyle{\bigoplus}_{j=1}^sP'_{y_j}$$ where $P'_x$, for $x \in
\Delta_0$, denotes the projective representation at $x$ in
$\rep^+(\Delta)$.  In particular, the support of ${\rm
top}(N_{\Delta})$ consists of the $y_j$. Recall that $\Lambda$ is
finite, contains a source vertex $a$, an arrow $a \to b$ and has the
property that every arrow in $Q$ starting in $a$ and supporting $M
\oplus N$ has an ending point in $\Lambda$.  Moreover, $a$ does not
lie in ${\rm supp}({\rm soc}(M \oplus N)_{\Lambda}) = {\rm
supp}({\rm soc}(N_{\Sigma_I}))$. Since $a$ is a source vertex in
$\Delta$ supporting $N_{\Delta}$, $a \in {\rm supp}({\rm top}
N_{\Delta})$, which means that $a = y_t$ for some $t$. For
simplicity, set $I_0 = \textstyle{\bigoplus}_{i=1}^rI'_{x_i}$ and
$I_1=\textstyle{\bigoplus}_{j=1}^sI'_{y_j}$. Let
$\alpha_1,\ldots,\alpha_q$ be the arrows in $\Delta$ starting in
$a$, where $\alpha_i : a \to b_i$, $i=1,2,\ldots,q$. Since $W$ is
indecomposable non-injective in $\rep(\Delta)$ and $a$ is a source
vertex in $\Delta$, $a \not \in {\rm supp}({\rm soc}W)$.
This means that no $x_i$ is equal to $a$. Thus, by the injectivity
of $I_0$, we get that
$$\sum_{i=1}^qI_0(\alpha_i): I_0(a) \to
\textstyle{\bigoplus}_{i=1}^qI_0(b_i)$$ is bijective. In particular,
$${\rm dim}_k(I_0(a)) = \sum_{i=1}^q {\rm dim}_k(I_0(b_i)).$$ Now, since $W_{\Lambda} = M_{\Lambda}$, $${\rm supp}({\rm soc}(W_{\Lambda})) \subseteq {\rm supp}({\rm
soc}(N_{\Lambda})) = {\rm supp}({\rm soc}(N_{\Sigma_I})).$$ By the
definition of the vertex $a$, $a \not \in {\rm supp}({\rm
soc}W_{\Lambda})$. Also, since $W_{\Lambda}$ is injective and every
non-zero $W(\alpha_i)$ is equal to $W_{\Lambda}(\alpha_i)$, one also
has
$${\rm dim}_k(W(a)) = \sum_{i=1}^q {\rm dim}_k(W(b_i)).$$  Therefore, using
$(*)$,
$${\rm dim}_k(I_1(a)) = \sum_{i=1}^q {\rm dim}_k(I_1(b_i)).$$
The last equality is true if and only if $a$ does not support the
socle of $I_1$, that is, if and only if $I'_a$ is not a direct
summand of $I_1$.  This means that $a \ne y_j$ for all $j$, a
contradiction to $a=y_t$. This proves the claim, that is,
$W_{\Lambda} \ne M_{\Lambda}$.

\smallskip

Suppose now that the support of $W$ is contained in $\Lambda$. By
restricting $\eta$ to ${\Sigma \backslash \Sigma_I}$, one gets
$(E')_{\Sigma \backslash \Sigma_I} \cong N_{\Sigma \backslash
\Sigma_I}$. This yields $M_{\Sigma \backslash \Sigma_I} = N_{\Sigma
\backslash \Sigma_I}$ since $N_{\Sigma \backslash \Sigma_I}$ is
indecomposable. But as observed above, $f_{\Sigma \backslash
\Sigma_I}$ is not an isomorphism. This contradiction shows that
${\rm supp}(W)$ is not contained in $\Lambda$. Let $L$ be the
sub-representation of $M_{\Delta}$ generated by $M_{\Sigma
\backslash \Sigma_I}$ and $W$. Since $W$ is finite dimensional, we
have $L \in \rep^+(\Delta)$. Since $M_{\Sigma \backslash \Sigma_I}$
is indecomposable, any proper decomposition of $L$ would yield a
proper section $s_L: L' \to L$ with $L'$ supported by $\Lambda$.
Hence $s_L$ factors through $W$, which means that $L'$ is a direct
summand of $W$. Since $W$ is indecomposable, $L' = W$, contradicting
${\rm supp}(W) \not \subseteq \Lambda$. This shows that $L$ is
indecomposable.
Hence, we have a proper inclusion $W \to L$ between indecomposable
representations. Since $W_{\Lambda} \ne M_{\Lambda}$, we have
another proper inclusion $L \to M_{\Delta}$ between indecomposable
representations of $\rep^+(\Delta)$. This contradicts the fact that
the inclusion $W \to M_{\Delta}$ is irreducible. \qedhere

\end{proof}

\section{The Auslander-Reiten quiver of $\rrep(Q)$}

In this section, we assume that $Q$ is a connected strongly locally
finite quiver and $\rrep(Q)$ is the full abelian subcategory of
$\rep(Q)$ of those objects being finite extension-representations of objects in $\rep^-(Q)$ by
objects in $\rep^+(Q)$. We give a complete description of the
Auslander-Reiten quiver of $\rrep(Q)$ by giving the possible shapes
of its connected components.
\smallskip

Let us recall some definitions. Let $\mathcal{C}$ be any skeletally
small abelian $k$-category such that every indecomposable object has
a local endomorphism algebra. We do not assume that $\C$ is
Hom-finite. Let us denote by ${\rm rad}_{\C}$ (or
simply ${\rm rad}$ when no risk of confusion) the ideal of $\C$
defined as follows. A morphism $f : X \to Y$ lies in ${\rm
rad}_{\C}(X,Y)$ if and only if, for every morphism $g : Y \to X$,
$\id_X-gf$ is an isomorphism.  Now, if $\id_X - gf$ is an
isomorphism of inverse $h$, then a straightforward argument yields
that $\id_Y - fg$ is an isomorphism of inverse $\id_Y + fhg$.
Hence, $f \in {\rm rad}(X,Y)$ if and only if, for every morphism $g:
Y \to X$, $\id_Y - fg$ is an isomorphism. The ideal ${\rm rad}_{\C}$
is known as the \emph{radical of $\C$} and a morphism $f \in {\rm
rad}(X,Y)$ is said to be a \emph{radical morphism}. When $\C$ is
Hom-finite, the description of the radical of $\C$ is given in
\cite{GR}.

It is well known that when $\C$ is Hom-finite with
$X,Y \in \C$, then $f:X \to Y$ is non-zero in $ {\rm
rad}_{\C}(X,Y)/{\rm rad}_{\C}^2(X,Y)$ if and only if $f$ is
irreducible.  This is also true in our setting when $X,Y$ are
indecomposable.

\begin{Lemma}
A morphism $f : X \to Y$ in $\C$ with $X,Y$ indecomposable is
irreducible if and only if it is a radical morphism whose class in $
{\rm rad}_{\C}(X,Y)/{\rm rad}_{\C}^2(X,Y)$ is non-zero.
\end{Lemma}

\begin{proof} Let $f : X \to Y$ with $X,Y$ indecomposable.  Assume
first that $f$ is irreducible.  Then $f$ is not an isomorphism.  For
$g: Y \to X$, $gf \in {\rm End}(X)$ is not an isomorphism, and hence
$\id_X - gf$ is an isomorphism since ${\rm End}(X)$ is a local
algebra. This shows that $f \in {\rm rad}(X,Y)$.  Suppose that $f
\in {\rm rad}^2(X,Y)$. Hence, $f = h_1g_1 + \cdots + h_rg_r$ where
$h_i : L_i \to Y$ and $g_i : X \to L_i$ are radical morphisms.  Let
$L = L_1 \oplus \cdots \oplus L_r$, $h=(h_1, \ldots, h_r)$ and $g =
(g_1, \ldots, g_r)^T$. Then $f = hg$ and $g$ is a section or $h$ is
a retraction.  Assume that $g$ is a section.  Then, for $1 \le i \le
r$, there exist $g_i':L_i \to X$ such that $g_1'g_1 + \cdots +
g_r'g_r = \id_X$. Since ${\rm End}(X)$ is local, this means that
$g_j'g_j$ is invertible (of inverse $q$) for some $j$, and hence
that $\id_X - qg_j'g_j$ is not invertible, contradicting the fact
that $g_j$ is a radical morphism. Hence, $g$ is not a section.
Similarly, $h$ is not a retraction. This shows that $f$ is non-zero
in ${\rm rad}(X,Y)/{\rm rad}^2(X,Y)$. Conversely, assume that $f$ is
non-zero in ${\rm rad}(X,Y)/{\rm rad}^2(X,Y)$.  It is clear that $f$
is not an isomorphism since $f \in {\rm rad}(X,Y)$.  Assume that $f
= hg$ with $g : X \to L$ and $h : L \to Y$.  By assumption, one of
$g,h$ is not a radical morphism. Assume that $g$ is not a radical
morphism. Hence, there exists $g' : L \to X$ such that $\id_X - g'g$
is not an isomorphism, meaning that $g'g$ is an isomorphism, since
${\rm End}(X)$ is local. But then $g$ is a section.  Similarly, if
$h$ is not a radical morphism, then $h$ is a retraction.
\end{proof}

Hence, for $X,Y$ indecomposable in $\C$, it makes sense to define
$${\rm irr}(X, Y):={\rm rad}(X, Y)/{\rm rad}^2(X, Y)$$ and call it the
set of irreducible maps from $X$ to $Y$ in $\C$. Let us now turn our
attention to the main object of study of the rest of the paper. The
\emph{Auslander-Reiten quiver} of $\mathcal{C}$, denoted $\it
\Gamma_{\mathcal{C}}$, is a partially valued translation quiver
defined as follows; compare \cite[(2.1)]{L3}. The vertex set is a
complete set of representatives of the isomorphism classes of the
indecomposable objects in $\mathcal{C}$. If $Z$ is a vertex of $\it
\Gamma_{\mathcal{C}}$, we denote by $k_Z$ the division $k$-algebra
${\rm End}(Z)/{\rm rad}(Z, Z)$. Let $X, Y$ be two vertices in
$\Ga_{\mathcal{C}}$.  By definition, there exists a unique arrow
$X\to Y$ in $\Ga_{\mathcal{C}}$ if and only if ${\rm irr}(X,Y)$ is
non-zero.  In this case, if ${\rm irr}(X,Y)$ is of finite length
over $k_X$ and $k_Y$, we attach to the arrow $X \to Y$ a valu\-ation
$(d_{_{XY}}, d'_{_{XY}})$ where $d_{_{XY}}'$ and $d_{_{XY}}$ are the
dimension of ${\rm irr}(X,Y)$ over $k_X$ and $k_Y$, respectively. In
this case, $d_{_{XY}}'$ and $d_{_{XY}}$ are the maximal integers
such that $\mathcal{C}$ admits an irreducible morphism
$X^{d'_{_{XY}}}\to Y$ and an irreducible morphism $X\to
Y^{d_{_{XY}}}$, respectively; see \cite[(3.4)]{Bau}. A valuation
$(d_{_{XY}}, d_{_{XY}}')$ is called {\it symmetric} if
$d_{_{XY}}=d_{_{XY}}'$, and {\it trivial} if
$d_{_{XY}}=d_{_{XY}}'=1$. For technical reasons, we replace each
arrow $X\to Y$ having a symmetric valuation $(d_{_{XY}}, d_{_{XY}})$
by $d_{_{XY}}$ unvalued arrows from $X$ to $Y$. The translation
$\tau$ is defined in such a way that $\tau Z=X$ if and only if
$\mathcal{C}$ has an almost split sequence
$$0 \to X \to Y \to Z \to 0.$$
Hence,
$\Ga_{\mathcal{C}}$ is actually a partially valued translation
quiver with multiple arrows in which all possible valuations are
non-symmetric.  If $\mathcal{C}$ is Hom-finite, then each arrow of
$\Ga_{\mathcal{C}}$ has a valuation attached to it (which is then
replaced by multiple arrows if it is symmetric). A connected
component of $\Ga_{\mathcal{C}}$ is called an {\it Auslander-Reiten
component} of $\mathcal{C}$.

\smallskip

In this section, we study the Auslander-Reiten quiver of $\rrep(Q)$,
which is a Hom-finite abelian $k$-category. We will
show that all arrows of $\Ga_{\rreps(Q)}$ have symmetric valuation,
and hence that $\Ga_{\rreps(Q)}$ is a quiver with no valuation. We
first need the following lemmas.

\begin{Lemma} \label{lemmairr}
An irreducible map between indecomposable objects of $\rrep(Q)$ is
irreducible in $\rep(Q)$.
\end{Lemma}

\begin{proof} Suppose that $f : M \to N$ is irreducible in $\rrep(Q)$ with
$M,N$ indecomposable.  From Lemma \ref{irred}, either $M \in
\rep^-(Q)$ or $N \in \rep^+(Q)$.  We only consider the first case,
that is, $M \in \rep^-(Q)$. If $M$ is injective in $\rep^-(Q)$, then
$h: M \to M/{\rm soc}M$ is a minimal left almost split map in
$\rep(Q)$ and hence also a minimal left almost split map in
$\rrep(Q)$. Hence, there exists a retraction $r: M/{\rm soc}M \to N$
such that $f=rh$ which shows that $f$ is irreducible in $\rep(Q)$.
Otherwise, we have an almost split sequence
$$\zeta: \;\; 0 \to M \stackrel{h}{\rightarrow} E \to L \to 0$$ in
$\rep(Q)$ with $L \in \rep^+(Q)$. In particular, $\zeta$ is almost
split in $\rrep(Q)$. Hence, there exists a retraction $r: E \to N$
such that $f=rh$ which shows that $f$ is irreducible in $\rep(Q)$.
\end{proof}

Recall from \cite{BLP} that an indecomposable representation $M \in
\rep^+(Q)$ is \emph{regular} in $\Ga_{\rep^+(Q)}$ if the connected
component of $\Ga_{\rep^+(Q)}$ containing $M$ does not contain a
representation of the form $P_x$ or $I_x$, $x \in Q_0$.

\begin{Lemma} \label{infdimrep2} Let $M$ be an infinite dimensional regular representation in $\Ga_{\rep^+\hspace{-1pt}(Q)}\vspace{1pt}$. Then
$\rep(Q)$ has a minimal right almost split morphism $h: N_1\oplus
N_2 \to M$ with $N_1$ indecomposable doubly-infinite and $N_2$
finite dimensional.

\end{Lemma}

\begin{proof}
There exists an almost split sequence
$$\eta: \;\; 0 \to M' \to E \to M \to 0$$
in $\rrep(Q)$ where $M'$ is finitely co-presented. Since $M$ is
infinite dimensional, $E$ is infinite dimensional. Let $L$ be an
infinite dimensional direct summand of $E$. There are irreducible
morphisms $f: M' \to L$ and $g: L \to M$ in $\rrep(Q)$. Let us first
assume that $M'$ is infinite dimensional. We claim that in this
case, there exists a left infinite path $p$ in ${\rm supp}(M')$ such
that $p \,\cap \, {\rm supp}(L)$ is infinite. Assume first that $f$
is an epimorphism. Since ${\rm supp}(L)$ is infinite and ${\rm
supp}(L) \subseteq {\rm supp}(M')$ with ${\rm supp}(M')$
socle-finite, there exists a left infinite path
$$p: \;\cdots \to x_3 \to x_2 \to x_1$$
in ${\rm supp}(M')$ such that infinitely many $x_i$ lie in the
support of $L$, showing the claim in this case. Suppose now that $f$
is a monomorphism.  Since $M'$ is infinite dimensional, ${\rm
supp}(M')$ contains a left infinite path $p$, and ${\rm supp}(M')
\subseteq {\rm supp}(L)$ yields the claim in this case.

Now, we show that $g$ is an epimorphism.  Suppose first that $M'$ is
finite dimensional.  Then $\eta$ is an almost split sequence in
$\rep^+(Q)$, and hence $L$ is finitely presented. Now, since
$\rep^+(Q)$ is Krull-Schmidt, $L$ has an indecomposable infinite
dimensional direct summand $L_0$ and from \cite[Lemma 4.13(2)]{BLP},
the restriction of $g$ to $L_0$ is an epimorphism.  In particular,
$g$ is an epimorphism. Suppose that $M'$ is infinite dimensional
while $g$ is a monomorphism. By the above claim, there exists a left
infinite path $p$ in ${\rm supp}(M')$ such that infinitely many
vertices of $p$ lie in ${\rm supp}(L) \subseteq {\rm supp}(M)$,
contradicting the fact that ${\rm supp}(M)$ is top-finite and $Q$ is
interval-finite. Hence, $g$ is an epimorphism.

Suppose now that $L_1,L_2$ are two infinite dimensional
representations such that $L_1 \oplus L_2$ is a direct summand of
$E$. Then we have epimorphisms $g_1: L_1 \to M$ and $g_2: L_2 \to
M$. Since $M$ is infinite dimensional, there exists a right infinite
path
$$y_1 \to y_2 \to y_3 \to \cdots$$
in ${\rm supp}(M)$. Since ${\rm supp}(M')$ is socle-finite, there
exists some $y_j$ with $y_j \not \in {\rm supp}(M')$. Then,
\begin{eqnarray*} {\rm dim}\, M(y_j) &=& {\rm
dim}\,(M')(y_j) + {\rm dim}\, M(y_j)\\ &=& {\rm dim}\, E(y_j)\\
&\ge& {\rm dim}\,L_1(y_j) + {\rm dim}\, L_2(y_j)\\ &\ge& {\rm dim}\,
M(y_j) + {\rm dim} \,M(y_j),
\end{eqnarray*}
which is a contradiction.  This shows that if $E=E_1 \oplus E_2$,
then one of $E_1,E_2$ is finite dimensional.  This proves the lemma.
\end{proof}

The following proposition says that the category $\rrep(Q)$ contains
\emph{most} of the Auslander-Reiten theory of $\rep(Q)$.

\begin{Prop} \label{propAR}
Let $\eta: \;\; 0 \to X \stackrel{u}{\rightarrow} Y
\stackrel{v}{\rightarrow} Z \to 0$ be a short exact sequence in
$\rep(Q)$. Then $\eta$ is almost split in $\rep(Q)$ if and only if
it is almost split in $\rrep(Q)$. In this case, $X \in \rep^-(Q)$,
$Z \in \rep^+(Q)$ and either
\begin{enumerate}[$(1)$] \item  $Y = Y_1 \oplus
    Y_2$ with $Y_1$ indecomposable doubly-infinite and $Y_2$ zero or
    indecomposable finite dimensional,
    \item The sequence lies in $\rep^-(Q)$,
    \item The sequence lies in $\rep^+(Q)$.
\end{enumerate}

\end{Prop}

\begin{proof}
Suppose that $\eta$ is almost split in $\rep(Q)$. By \cite[Theorem
3.5]{Pa}, $X \in \rep^-(Q)$ and $Z \in \rep^+(Q)$. By Lemma
\ref{Ext+-}, $\eta$ lies in $\rrep(Q)$ and hence is almost split in
$\rrep(Q)$. Suppose now that $\eta$ is almost split in $\rrep(Q)$.
Then $X,Z$ are (strongly) indecomposable.  If $Z \in \rep^+(Q)$,
there is an almost split sequence
$$\xi: \quad 0 \to X' \to Y' \to Z \to 0$$
in $\rep(Q)$ which lies in $\rrep(Q)$ by what we have shown. By the
unicity of almost split sequences, $\xi = \eta$ and we are done. We
can treat similarly the case where $X \in \rep^-(Q)$. Assume now
that $X \not \in \rep^-(Q)$ and $Z \not \in \rep^+(Q)$. By
Proposition \ref{irred}, we must have that $Y \in \rep^+(Q) \cap
\rep^-(Q)$, which is impossible.  This shows that $\eta$ is almost
split in $\rep(Q)$.

Assume now that $\eta$ is almost split in $\rrep(Q)$ (and hence in
$\rep(Q)$). If $X$ is finite dimensional, then (3) holds and if $Z$
is finite dimensional, then (2) holds. Otherwise, since both $X,Z$
are infinite dimensional, $Z$ is regular in $\Ga_{\rep^+(Q)}$ and
from Lemma \ref{infdimrep2}, $Y=Y_1 \oplus Y_2$ with $Y_1$
indecomposable infinite dimensional and $Y_2$ finite dimensional. If
$Y_1 \in \rep^+(Q)$, then $Y \in \rep^+(Q)$ and hence $X \in
\rep^+(Q)$. Being in $\rep^-(Q)$, we get $X \in \rep^b(Q)$, a
contradiction. Hence, $Y_1 \not \in \rep^+(Q)$. Similarly, $Y_1 \not
\in \rep^-(Q)$. Therefore, $Y_1$ is doubly-infinite. Now, if $Y_2$
is non-zero, then there is an irreducible map $Y_2 \to Z$ in
$\rep^+(Q)$.  From \cite[Theorem 4.14]{BLP} (see also Theorem
\ref{regcomponentgen+}), $Y_2$ is indecomposable. This proves that
one of (1), (2) or (3) hold.
\end{proof}

Unfortunately, there may be irreducible maps $M \to N$ in $\rep(Q)$
with $M,N$ indecomposable but not in $\rrep(Q)$.  Hence, the
Auslander-Reiten quiver of $\rrep(Q)$ misses some irreducible maps
of $\rep(Q)$.  However, in the next section, we shall see that these
irreducible morphisms are \emph{isolated} from the irreducible
morphisms in $\rrep(Q)$.

\begin{Exam}
Let $Q$ be the following quiver
$$\xymatrixcolsep{10pt}\xymatrixrowsep{2pt}\xymatrix{& 5\ar@{.}[dl] \ar[dr] && 3\ar[dl] \ar[dr] && 1\ar[dl] \ar[dr]
& \\
&& 4 && 2 && 0}$$ of type $\A_{\infty}$ with zigzag orientation.
Consider the indecomposable sincere representation $M$ such that
$M(i)=k$ for all $i \in \N$.  Let $N$ be the quotient of $M$ by the
simple representation at the vertex $0$. Then the morphism $M \to N$
is irreducible in $\rep(Q)$ with $M,N \not \in \rrep(Q)$.  Observe,
however, that there is no almost split sequence in $\rep(Q)$
starting or ending in $M$ or $N$, by Proposition \ref{propAR}.
\end{Exam}

\begin{Lemma} \label{IrrValuations}
Let $f: M \to N$ be a morphism in $\rep^+(Q)$.  Then $f$ is
irreducible in $\rep^+(Q)$ if and only if $f$ is irreducible in
$\rrep(Q)$.
\end{Lemma}

\begin{proof} We only need to prove the necessity. Suppose that $f$
is irreducible in $\rep^+(Q)$. Let $L \in \rrep(Q)$ with two
morphisms $u:M \to L$, $v: L \to N$ such that $f=vu$. Let $S$ be the
set of vertices $x$ in $Q$ such that there exists an arrow $\alpha:
x \to y$ with $x \not \in {\rm supp}(M \oplus N)$, $y \in {\rm
supp}(M \oplus N)$ and $L(\alpha) \ne 0$. Since $L \in \rrep(Q)$,
there exists a top-finite successor-closed subquiver $\Omega$ of
${\rm supp}(L)$ such that $L_{\Omega} \in \rep^+(Q)$ and there is a
finite number of arrows $\beta: a \to b$ with $a \in {\rm supp}(L)
\backslash \Omega$, $b \in \Omega$ and $L(\beta) \ne 0$.
The successor-closed subquiver $\Omega'$ of $Q$ generated by
$\Omega$ and ${\rm supp}(M \oplus N)$ is top-finite and is such that
$L_{\Omega'} \in \rep^+(Q)$. By Lemma \ref{FiniteExtLemma}, since $L
\in \rrep(Q)$, $L$ is a finite extension-representation of $L/L_{\Omega'}$ by
$L_{\Omega'}$.  In particular, $S \cap ({\rm supp}(L) \backslash
\Omega')$ is finite. Therefore, the successor-closed subquiver
$\Sigma$ of $Q$ generated by $S$ and $\Omega'$ is top-finite with
$L_{\Sigma} \in \rep^+(Q)$.  Thus, we have a factorization $f =
v_{\Sigma}u_{\Sigma}$ in $\rep^+(Q)$. Therefore, $u_{\Sigma}$ is a
section or $v_{\Sigma}$ is a retraction. Since $\Sigma$ is
successor-closed in $Q$ and contains the vertices in $S$, one easily
checks that $u_{\Sigma}$ is a section if and only if $u$ is a
section; and $v_{\Sigma}$ is a retraction if and only if $v$ is a
retraction. This shows that $f$ is irreducible in $\rrep(Q)$.
\end{proof}

Let $\Ga$ be a connected component of $\Ga_{\rreps(Q)}$.  Then $\Ga$
is said to be {\it preprojective} if it contains a projective object
in $\rep^+(Q)$ and {\it preinjective} if it contains an injective
object in $\rep^-(Q)$. Otherwise, it is called {\it regular}. A full
convex and connected subquiver $\Delta$ of $\Ga$ is a {\it section}
if it contains no oriented cycle and meets every $\tau$-orbit of
$\Ga$ exactly once. It is \emph{right-most} if $\tau X$ is not
defined for every $X \in \Delta$; and \emph{left-most} if $\tau^- X$
is not defined for every $X \in \Delta$.

\begin{Theo} \label{preproj}
Let $Q$ be connected infinite and strongly locally finite. Then
$\Ga_{\rreps(Q)}$ contains a unique preprojective component
$\mathcal{P}_Q$ having a left-most section $P_Q$ consisting of all
the indecomposable projective objects in $\rep^+(Q)$.  Moreover,
\begin{enumerate}[$(1)$]
    \item If
$Q$ has no right infinite path, then $\mathcal{P}_Q$ is of shape to
$\N Q^{\rm\,op}$.
    \item Otherwise, it
is a predecessor-closed subquiver of $\N Q^{\rm\,op}$ having a
right-most section consisting of the infinite dimensional
representations of $\mathcal{P}_Q$.
\end{enumerate}
\end{Theo}

\begin{proof}
The statement has been proven for the category $\rep^+(Q)$ in
\cite{BLP}. Let $\Ga$ be the unique preprojective component of
$\Ga_{\rep^+(Q)}$ and $X \in \Ga$. If $X$ is not projective in
$\rep^+(Q)$, then one has an almost split sequence
$$\eta: \;\; 0 \to X' \to E \to X \to 0$$
in $\rep(Q)$.  Since $X$ is preprojective in $\rep^+(Q)$, $X'$ is
finite dimensional and the sequence is almost split in $\rep^+(Q)$
and also in $\rrep(Q)$. In particular, we have a minimal right
almost split morphism $E \to X$ in $\rep^+(Q)$ which is also minimal
right almost split in $\rrep(Q)$.  This will also be the case if $X$
is projective in $\rep^+(Q)$. Suppose first that we have an arrow
$\alpha: Y \to X$ in $\Ga_{\rreps(Q)}$. Using what we just proved,
$Y \in \rep^+(Q)$ and we get an arrow $\alpha': Y \to X$ in $\Ga$.
The valuations of $\alpha$ and $\alpha'$ need to coincide by Lemma
\ref{IrrValuations}. Hence, $\Ga$ is a predecessor-closed subquiver
of $\Ga_{\rreps(Q)}$.

Suppose now that we have an arrow $\beta: X \to Y$ in
$\Ga_{\rreps(Q)}$.  We have an irreducible map $f: X \to Y$ in
$\rrep(Q)$, which needs to be irreducible in $\rep^+(Q)$ by
Proposition \ref{irred}. Therefore, we have an arrow $\beta': X \to
Y$ in $\Ga$, and the valuations of $\beta$ and $\beta'$ coincide by
Lemma \ref{IrrValuations}. This shows that $\Ga$ is a
successor-closed subquiver of $\Ga_{\rreps(Q)}$.  Therefore, $\Ga$
is a connected component of $\Ga_{\rreps(Q)}$, and consequently,
since it contains all the $P_x$, $x \in Q_0$, is the unique
preprojective component of $\Ga_{\rreps(Q)}$.
\end{proof}

A dual argument yields the following dual result for the
preinjective component.

\begin{Theo} \label{preinj}
Let $Q$ be connected infinite and strongly locally finite. Then
$\Ga_{\rreps(Q)}$ contains a unique preinjective component
$\mathcal{I}_Q$ having a right-most section $I_Q$ consisting of all
the indecomposable injective objects in $\rep^-(Q)$. Moreover,
\begin{enumerate}[$(1)$]
    \item If
$Q$ has no left infinite path, then $\mathcal{I}_Q$ is isomorphic to
$\N^- Q^{\rm\,op}$.
    \item Otherwise, it
is a successor-closed subquiver of $\N^- Q^{\rm\,op}$ having a
left-most section consisting of the infinite dimensional
representations of $\mathcal{I}_Q$.
\end{enumerate}
\end{Theo}

Recall from \cite{BLP} or \cite{Ri} that a valued translation quiver
is said to be of (finite) {\it wing type} if it is isomorphic to the
following translation quiver with trivial valuations\,:
$$\xymatrixrowsep{16pt} \xymatrixcolsep{20pt}
\xymatrix@!=0.1pt{&&&& \circ \ar[dr]&&&&\\
&&&\circ \ar[ur]\ar[dr] && \circ \ar[dr]&&&\\
&& \circ\ar[ur]  && \circ\ar[ur]  && \circ
&&\\
}$$ \vspace{-15 pt}
$$\iddots \;\;\ddots \;\; \iddots \;\;\;  \ddots \; \;\iddots\;\;\;\, \ddots$$
\vspace{-21 pt}
$$\xymatrixrowsep{16pt}
\xymatrixcolsep{20pt}\xymatrix@!=0.1pt{& \circ\ar[dr]  && \circ   &\cdots& \circ\ar[dr]   && \circ \ar[dr] & \\
\circ \ar[ur] && \circ \ar[ur]  &\cdots&&\cdots& \circ\ar[ur]
 && \circ}\vspace{6pt}
$$

The following theorem was proven in \cite{BLP}.  An indecomposable
representation $M$ in $\rep^+(Q)$ is \emph{pseudo-projective} if it
is not projective and the almost split sequence
$$0 \to M' \to E \to M \to 0$$
in $\rep(Q)$ is such that $M'$ is infinite dimensional.

\begin{Theo}[Bautista, Liu and Paquette]\label{regcomponentgen+} Let $Q$ be an infinite, connected and strongly locally finite quiver. Let $\Ga$ be a regular component of $\Ga_{\rep^+(Q)}$.

\begin{enumerate}[$(1)$]
\item If $\Ga$ has no infinite dimensional or pseudo-projective representation, then it is of shape $\mathbb{Z} \mathbb{A}_\infty$.
\item If $\Ga$ has infinite dimensional but no pseudo-projective representations, then it is of shape $\mathbb{N}^- \mathbb{A}_\infty$ and its right-most section is a left infinite path.
\item If $\Ga$ has pseudo-projective but no infinite dimensional representations, then it is of shape $\mathbb{N}\mathbb{A}_\infty$ and its left-most section is a right infinite path.
\item  If $\Ga$ has both pseudo-projective and infinite dimensional representations, then $\Ga$ is finite of wing type.
\end{enumerate}
\end{Theo}

We have a similar theorem for the category $\rrep(Q)$.

\begin{Theo}\label{regcomponentgen} Let $Q$ be a connected infinite and strongly locally
finite quiver. Let $\Ga$ be a regular component of $\Ga_{\rreps(Q)}$.
\begin{enumerate}[$(1)$]
\item If $\Ga$ has no infinite dimensional representation, then it is of shape $\mathbb{Z} \mathbb{A}_\infty$.
\item If $\Ga$ has infinite dimensional representations all lying in $\rep^+(Q)$, then it is of shape $\mathbb{N}^-\hskip -0.5pt \mathbb{A}_\infty$ and its right-most section is a left infinite path.
\item If $\Ga$ has infinite dimensional representations all lying in $\rep^-(Q)$, then it is of shape $\mathbb{N}\mathbb{A}_\infty$ and its left-most section is a right infinite path.
\item Otherwise, $\Ga$ is finite of wing type and contains exactly one doubly-infinite representation.
\end{enumerate}
\end{Theo}

\begin{proof}
If $\Ga$ contains only finite dimensional representations, then for
any $X \in \Ga$, one has almost split sequences
$$0 \to X' \to E \to X \to 0$$
and
$$0 \to X \to E' \to X'' \to 0$$
in $\rrep(Q)$ which are almost split in $\rep^+(Q)$. Therefore,
$\Ga$ is a component of the Auslander-Reiten quiver of $\rep^+(Q)$
by Lemma \ref{IrrValuations}.  Hence, $\Ga$ is of shape
$\Z\A_{\infty}$ by Theorem \ref{regcomponentgen+}.

Suppose now that $\Ga$ has infinite dimensional representations all
lying in $\rep^+(Q)$. Then for any $X \in \Ga$, the almost split
sequence
$$0 \to X' \to E \to X \to 0$$
in $\rrep(Q)$ is an almost split sequence in $\rep^+(Q)$. Moreover,
by Proposition \ref{irred}, every irreducible morphism $X \to Y$ in
$\rrep(Q)$ with $Y$ indecomposable is such that $Y \in \rep^+(Q)$.
This shows that $\Ga$ is a (predecessor-closed) connected component
of the Auslander-Reiten quiver of $\rep^+(Q)$ by Lemma
\ref{IrrValuations}. Hence $\Ga$ is of shape $\N^-\A_{\infty}$ and
its right-most section is a left infinite path by Theorem
\ref{regcomponentgen+}.

Dually, if $\Ga$ has infinite dimensional representations all lying
in $\rep^-(Q)$, then $\Ga$ is of shape $\mathbb{N}\mathbb{A}_\infty$
and its left-most section is a right infinite path.

Consider now the case where $\Ga$ contains an infinite dimensional
representation in $\rep^+(Q)$ and an infinite dimensional
representation in $\rep^-(Q)$.  From Proposition \ref{irred}, we get
that the full subquiver $\Ga'$ of $\Ga$ consisting of the
representations in $\rep^+(Q)$ is successor-closed in $\Ga$.
Similarly, the full subquiver $\Ga''$ of $\Ga$ consisting of the
representations in $\rep^-(Q)$ is predecessor-closed in $\Ga$. Now,
by Lemma \ref{IrrValuations}, $\Ga'$ is a connected component of the
Auslander-Reiten quiver of $\rep^+(Q)$. If $\Ga'$ is left stable as
a translation quiver, then we see that $\Ga = \Ga'$, a
contradiction.  Hence, $\Ga'$ is not left stable and contains
infinite dimensional representations in $\rep^+(Q)$. Since $\Ga'$
does not contain any of the $P_x$ and $I_x$, $x \in Q_0$, we see
that $\Ga'$ is a regular component of the Auslander-Reiten quiver of
$\rep^+(Q)$. Hence, $\Ga'$ is of wing type with trivial valuations
by Theorem \ref{regcomponentgen+}. We get similarly that $\Ga''$ is
of wing type. Hence, $\Ga' \cup \Ga''$ is a full subquiver of $\Ga$
which has trivial valuations and which is of the form
$$\xymatrixrowsep{22pt} \xymatrixcolsep{25pt}
\xymatrix@!=0.1pt{
&&&X_{n,1} \ar[dr] && X_{n,2} \ar[dr]&&&\\
&& X_{n-1,1} \ar[ur]  && X_{n-1,2} \ar[ur]  && X_{n-1,3}
&&\\
}$$ \vspace{-10 pt}
$$\iddots \;\;\;\;\ddots \;\;\;\; \iddots \;\;\;\;\;  \ddots \;\;\; \;\iddots\;\;\;\;\;\, \ddots$$
\vspace{-21 pt}
$$\;\;\xymatrixrowsep{22pt}
\xymatrixcolsep{25pt}\xymatrix@!=0.1pt{& X_{2,1}\ar[dr]  && X_{2,2}   &\hspace{-2 pt}\cdots& X_{2,n-1}\ar[dr]   && X_{2,n}\ar[dr] & \\
X_{1,1} \ar[ur] && X_{1,2} \ar[ur]  &\cdots&&\cdots& X_{1,n}\ar[ur]
 && X_{1,n+1}}\vspace{6pt}
$$
where the $X_{1,1}, \ldots, X_{n,1}$ are the infinite dimensional
representations in $\Ga''$ and $X_{n,2},X_{n-1,3},\ldots,X_{1,n+1}$
are the infinite dimensional representations in $\Ga'$. By
Proposition \ref{irred}, the vertices in $\Ga \backslash (\Ga' \cup
\Ga'')$ are all doubly-infinite representations. Since $\Ga'$ is
successor-closed in $\Ga$ and $\Ga''$ is predecessor-closed in
$\Ga$, the arrows in $\Ga$ attached to a vertex in $\Ga' \cup \Ga''$
and which are not in $\Ga' \cup \Ga''$ start in $X_{n,1}$ or end in
$X_{n,2}$. By Proposition \ref{propAR}, since $X_{n,1}, X_{n,2}$ are
infinite dimensional, there is an almost split sequence
$$\eta: \;\; 0 \to X_{n,1} \to E \to X_{n,2} \to 0$$
in $\rrep(Q)$ where $E \cong X_{n-1,2} \oplus E' $ and $E'$ is
doubly infinite.  By Proposition \ref{propAR}, it is clear that the
only arrow ending in $E'$ is $X_{n,1} \to E'$ and the only arrow
starting in $E'$ is $E' \to X_{n,2}$. Hence, $\Ga$ is of shape

$$\xymatrixrowsep{22pt} \xymatrixcolsep{25pt}
\xymatrix@!=0.1pt{&&&& \;E'\ar[dr] &&&&\\
&&&X_{n,1} \ar[dr] \ar[ur]&& X_{n,2} \ar[dr]&&&\\
&& X_{n-1,1} \ar[ur]  && X_{n-1,2} \ar[ur]  && X_{n-1,3}
&&\\
}$$ \vspace{-11 pt}
$$\iddots \;\;\;\;\ddots \;\;\;\; \iddots \;\;\;\;\;  \ddots \;\;\; \;\iddots\;\;\;\;\;\, \ddots$$
\vspace{-18 pt}
$$\;\xymatrixrowsep{22pt}
\xymatrixcolsep{25pt}\xymatrix@!=0.1pt{& X_{2,1}\ar[dr]  && X_{2,2}   &\hspace{-2 pt}\cdots& X_{2,n-1}\ar[dr]   && X_{2,n}\ar[dr] & \\
X_{1,1} \ar[ur] && X_{1,2} \ar[ur]  &\cdots&&\cdots& X_{1,n}\ar[ur]
 && X_{1,n+1}}\vspace{6pt}
$$
and it remains to show that the arrows $X_{n,1} \to E'$ and $E' \to
X_{n,2}$ are trivially valued. Let $(e,d)$ be the valuation of $E'
\to X_{n,2}$.  It is clear that $d=1$. Suppose that there is an
irreducible map $f: E' \to X_{n,2}^r$, $r \ge 2$, in $\rrep(Q)$.
Since ${\rm supp}(E')$ contains ${\rm supp}(X_{n,2})$ properly, $f$
is an epimorphism. Let
$$p: \;\; x_0 \to x_1 \to \cdots$$ be a right infinite
path in ${\rm supp}(X_{n,2})$.  Since ${\rm supp}(X_{n,1})$ is
socle-finite and $X_{n-1,2}$ is finite dimensional, there exists an
integer $j$ with $x_j \not \in {\rm supp}(X_{n,1} \oplus
X_{n-1,2})$. The almost split sequence $\eta$ then yields ${\rm
dim}E'(x_j) = {\rm dim}X_{n,2}(x_j) \ne 0$. However, since $f$ is an
epimorphism,
$${\rm dim}X_{n,2}(x_j) = {\rm dim}E'(x_j) \ge 2{\rm dim}X_{n,2}(x_j),$$
a contradiction.  Hence, we get that the valuation of the arrow $E'
\to X_{n,2}$ in $\Ga$ is $(1,1)$.  Similarly, the valuation of the
arrow $X_{n,1} \to E'$ in $\Ga$ is $(1,1)$.  This shows that $\Ga$
is of wing type.

If $\Ga$ does not contain representations in $\rep^+(Q) \cup
\rep^-(Q)$, then it contains only doubly-infinite representations.
By Proposition \ref{irred}, $\Ga$ is a trivial component, and hence
is necessarily of wing type. \qedhere

\end{proof}

\begin{Remark}
(1) The last theorem says in particular that there exists a
bijection between the isomorphism classes of doubly-infinite
representations and the regular components of wing type in
$\Ga_{\rreps(Q)}$.

\smallskip

\noindent (2) If $Q$ is connected strongly locally finite, then
there exists a doubly-infinite representation in $\rrep(Q)$ if and
only if $Q$ has a left-infinite path and a right-infinite path.
Hence, if $\Ga_{\rreps(Q)}$ has both a regular component of shape
$\N \A_\infty$ and of shape $\N^-\A_\infty$, then it needs to have a
regular component of wing type.

\smallskip

\noindent (3) The Auslander-Reiten quiver of $\rep^+(Q)$ is a
successor-closed subquiver of $\Ga_{\rreps(Q)}$.  It is obtained by
removing the infinite dimensional non-finitely presented
representations in $\Ga_{\rreps(Q)}$.  Similarly, the
Auslander-Reiten quiver of $\rep^-(Q)$ is a predecessor-closed
subquiver of $\Ga_{\rreps(Q)}$.  It is obtained by removing the
infinite dimensional non-finitely co-presented representations in
$\Ga_{\rreps(Q)}$.

\smallskip

\noindent (4) The Auslander-Reiten quiver of $\rrep(Q)$ has a
symmetric valuation by Theorems \ref{preproj}, \ref{preinj} and
\ref{regcomponentgen}.

\smallskip

\noindent (5) If $M$ is doubly-infinite and has no simple submodule
and no simple quotient, then $\{M\}$ is a trivial component of
$\Ga_{\rreps(Q)}$.
\end{Remark}

An additive Krull-Schmidt $k$-category $\mathcal{C}$ is said to be a
\emph{left Auslander-Reiten category} if every indecomposable object
in $\mathcal C$ is the domain of a minimal left almost split
epimorphism or is the starting term of an almost split sequence; a
\emph{right Auslander-Reiten category} if every indecomposable
object in $\mathcal C$ is the co-domain of a minimal right almost
split monomorphism or the ending term of an almost split sequence;
and an \emph{Auslander-Reiten category} if it is a left and a right
Auslander-Reiten category; compare \cite[(2.6)]{L3}.

\smallskip

The following proposition follows easily from our previous results.

\begin{Prop} Let $Q$ be a strongly locally finite quiver.
\begin{enumerate}[$(1)$]
    \item The category $\rrep(Q)$ is a left Auslander-Reiten category if and only
    if $Q$ has no right-infinite path.
    \item The category $\rrep(Q)$ is a right Auslander-Reiten category if and only
    if $Q$ has no left-infinite path.
    \item The category $\rrep(Q)$ is an Auslander-Reiten category if and only
    if $Q$ has no infinite path.
\end{enumerate}
\end{Prop}

\section{The Auslander-Reiten quiver of $\rep(Q)$}

In this section, again, $Q$ stands for a connected strongly locally
finite quiver.  Although $\rep(Q)$ is, in general, not
Hom-finite, it is true that every indecomposable
object in $\rep(Q)$ has a local endomorphism algebra; see\cite{GR}.
 Thus, one can construct the Auslander-Reiten quiver $\Ga_{\rep(Q)}$ of
$\rep(Q)$ as defined in Section $5$.
The objective of this section is to show that the Auslander-Reiten
components of $\rrep(Q)$ are connected components of
$\Ga_{\rep(Q)}$.

\smallskip

We start with the following lemma, where the proof is inspired from
the proof of \cite[Proposition 2.1]{Pa}.

\begin{Lemma} \label{lemmasection}
Let $N \in \rep(Q)$ with a sub-representation $M$ and suppose we
have a chain
$$0 = L_0 \subset L_1 \subset L_2 \subset \cdots$$
of finitely generated proper sub-representations of $N$ with every
inclusion $M \to M+L_i$ being a section.
Suppose moreover that the union of the $M + L_i$ is $N$. Then the
inclusion $M \to N$ is a section.
\end{Lemma}

\begin{proof}
Set $M_i = M + L_i$ for $i \ge 0$. For $0 \le i < j$, denote the
inclusion $M \to M_i$ by $q_i$ and the inclusion $M_i \to M_j$ by
$q_{i,j}$. For $i \ge 0$, we have a short exact sequence
$$0 \to M \cap L_i \to M \oplus L_i \to M_i \to 0$$
giving a monomorphism
$$\varphi_i : \Hom(M_i,M) \to \Hom(M,M) \oplus
\Hom(L_i,M),$$ sending a morphism $g$ to $(g_1,g_2)$ where $g_1$ is
the restriction of $g$ to $M$ and $g_2$ is the restriction of $g$ to
$L_i$. Let $V_i$ denote the subspace of $\Hom(M_i,M)$ of the
morphisms $g$ for which $gq_i$ is a scalar multiple of $\id_M$.
Since $V_i$ is the pre-image of $k\langle\id_M\rangle \oplus
\Hom(L_i,M)$, which is finite dimensional, we see that $V_i$ is
finite dimensional. A morphism $g \in V_i$ for which $gq_i=\id_M$ is
called \emph{normalized}. By assumption, each $V_i$ contains a
normalized morphism and hence is non-zero. Now, one has a non-zero
map
$$g_i : V_{i+1} \to V_i$$
which is induced by $q_{i,i+1}$ and sends a normalized map to a
normalized one. By assumption, we have a normalized map $v_i : M_i
\to M$ in $V_i$ such that
$$v_iq_{j,i} = v_iq_{i-1,i}q_{i-2,i-1}\cdots q_{j,j+1} = g_{j}\cdots g_{i-1}(v_i)$$ is
normalized in $V_{j}$ for $0 \le j < i$.  Let $0\ne M_{ij} = {\rm
Im}(g_jg_{j+1}\cdots g_{i-1})$ for $0 \le j < i$ with $M_{ii} =
V_i$. The chain
$$M_{jj} \supseteq M_{j+1,j} \supseteq M_{j+2,j} \cdots$$
of finite dimensional $k$-vector spaces yields an integer $r_j \ge
j$ for which $0 \ne M_{r_j,j} = M_{k,j}$ whenever $k \ge r_j$.
Moreover, each such $M_{r_j,j}$ contains a normalized map.  Then the
maps $g_i$ clearly induce non-zero maps $$\overline{g}_i :
M_{r_{i+1},i+1} \to M_{r_i,i}.$$  We claim that these maps are
surjective.  Let $u \in M_{r_i,i}$.  For every positive integer
$r>i+1$, $u \in {\rm Im}(g_ig_{i+1} \cdots g_{r-1})$ and hence,
there exists an element $u_r \in {\rm Im}(g_{i+1} \cdots g_{r-1})$
such that ${g_i}(u_r)=u$.  But then $u_{r_{i+1}} \in
M_{r_{i+1},i+1}$ is such that $\overline{g_i}(u_{r_{i+1}})=u$,
showing the claim.  Now, set $u_0 \in M_{r_0,0}$ be a normalized
map.  Then there exists $u_1 \in M_{r_1,1}$ such that
$\overline{g}_0(u_1) = u_0$.  Observe that if $u_1$ is not
normalized, then there exists $\alpha \in k\backslash\{0\}$ such
that $\alpha u_1$ is normalized and hence that
$\overline{g}_0(\alpha u_1)=\alpha u_0$ is normalized, showing that
$\alpha = 1$.  Hence, $u_1$ is normalized.  Choose such $u_i \in
M_{r_i,i}$ for all positive integers $i$.  Hence, for $i \ge 0$, we
have that $u_iq_i=\id_M$ and $u_{i+1}q_{i,i+1}=u_i$. Since $N$ is
the union of the $M_i$, it is also the direct limit of the $M_i$.
Therefore, the family of morphisms $u_i : M_i \to M$ yields a unique
morphism $h : N \to M$ such that $hr_i= u_i$ for $i\ge 0$, where
$r_i : M_i \to N$ is the inclusion. This shows that the inclusion $M
\to N$ is a section.
\end{proof}

As an immediate consequence of the preceding lemma, we get the
following.

\begin{Cor} \label{Corirred}
Let $f: M \to N$ be an irreducible monomorphism with $M,N \in
\rep(Q)$.  Then $N = {\rm Im}(f) + L$ where $L$ is finitely
generated. In particular, ${\rm Coker}(f)$ is finitely generated
indecomposable.
\end{Cor}

\begin{proof}
Since $Q$ is connected and strongly locally finite, it has a
countable number of vertices and we can find, for $i \ge 0$, a chain
$$0 = L_0 \subseteq L_1 \subseteq L_2 \subseteq \cdots $$
of finitely generated sub-representations of $N$ such that the union
of the $M + L_i$ is equal to $N$.  If the above chain is not
stationary, then using Lemma \ref{lemmasection}, we get that $f$ is
a section, which is impossible. This shows the first part of the
statement. The fact that ${\rm Coker}(f)$ is indecomposable follows
from the well known result that the cokernel of an irreducible
monomorphism in an abelian category is indecomposable; see
\cite{AuR}.
\end{proof}

\begin{Lemma}
Let $f: M \to N$ be an irreducible monomorphism with $M \in
\rrep(Q)$ and $N \in \rep(Q)$.  Then $N \in \rrep(Q)$.
\end{Lemma}

\begin{proof}
We may assume that $f$ is the inclusion. By Corollary
\ref{Corirred}, $N = M + L$ where $L$ is finitely generated.  In
particular, ${\rm supp}(L)$ is top-finite.

Let $\Sigma$ be the support of $N$.  Since $M \in \rrep(Q)$ and $N =
M+L$, there exists a top-finite successor-closed subquiver $\Omega$
of $\Sigma$ such that $M$ is a finite extension-representation of
$M/M_{\Omega} \in \rep^-(Q)$ by $M_{\Omega}\in \rep^+(Q)$ and
$\Omega$ contains the support of $L$. Moreover, by Proposition
\ref{PropProjFiniteInj}, there exists a co-finite successor-closed
subquiver $\Sigma_P$ of $\Omega$ such that $M_{\Sigma_P}$ is
projective in $\rep^+(Q)$. Take $\Theta$ to be the successor-closed
subquiver of $Q$ generated by $\Sigma_P$. We still have that
$M_{\Theta} \in \rep^+(Q)$ is projective and $M$ is a finite
extension-representation of $M/M_{\Theta}$ by $M_{\Theta}$. Using
this and since $N=M+L$ with ${\rm supp}(L)$ top-finite, there exists
a co-finite successor-closed subquiver $\Sigma'$ of $\Theta$ such
that every arrow $x \to y$ supporting $N$ with $y \in \Sigma'$ is
such that $x \in \Theta$. We want to show that $N_{\Omega}$ is
finitely presented, or equivalently, that $N_{\Theta}$ is finitely
presented. Suppose it is not the case. First, since $L$ is finitely
generated, we see that $L_{\Theta}$ is finitely generated since the
inclusion $L_{\Theta} \to L$ has a finite dimensional cokernel.
Therefore, $N_{\Theta}$ is finitely generated since $N_{\Theta} =
M_{\Theta} + L_{\Theta}$. Hence, we have a projective resolution of
the form
$$0 \to \textstyle{\bigoplus}_{i=1}^{\infty}P_{y_i} \stackrel{h}{\rightarrow} \textstyle{\bigoplus}_{j=1}^{r}P_{x_j} \to N_{\Theta} \to 0$$
where $h$ is a radical morphism and all $y_i$ lie in $\Theta$ since
$\Theta$ is successor-closed. There exists infinitely many $i$ with
$y_i \in \Sigma'$. Set $I$ to be the set of all such $i$. For each
$i \in I$, the pushout of the the projection
$\textstyle{\bigoplus}_{j=1}^{\infty}P_{y_j} \to S_{y_i}$ with $h$
then yields an exact sequence
$$0 \to S_{y_i} \to E_i \stackrel{v_i}{\rightarrow} N_{\Theta} \to 0$$
which is non-split since $h$ is a radical morphism. Let $E_i'$ be
the following representation. For $x \in Q_0$, we set
$$(E_i')(x) = \left\{%
\begin{array}{ll}
    E_i(x), & \hbox{if $x \in \Theta_0$;} \\
    N(x), & \hbox{otherwise.} \\
\end{array}%
\right.    $$ and for $\alpha \in Q_1$,
$$(E_i')(\alpha) = \left\{%
\begin{array}{ll}
    E_i(\alpha), & \hbox{if $\alpha \in \Theta_1$;} \\
    N(\alpha), & \hbox{if $\alpha: x \to y$ with $x,y \not \in \Theta_1$;} \\
    (v_i)_{y}^{-1}N(\alpha), & \hbox{if $\alpha: x \to y$ with $x \not \in \Theta_1$ and $y \in \Theta_1$, $y \ne y_j$;} \\
    0, & \hbox{otherwise.} \\
\end{array}%
\right.    $$
Since every arrow attached to $y_i$ and supporting $N$ is clearly in
$\Theta$, we can extend the last extension to a non-split short
exact sequence
$$0 \to S_{y_i} \to E_i' \stackrel{v_i'}{\rightarrow} N \to 0$$
where $(v_i')_x = (v_i)_x$ if $x \in \Theta_0$ and $(v_i')_x =
\id_{N(x)}$, otherwise.
Hence, we have a pullback diagram
$$\xymatrixrowsep{16pt} \xymatrixcolsep{20pt}\xymatrix@1{
0 \ar[r] & S_{y_i} \ar[r] \ar@{=}[d] &  F_i \ar[d]^{u_i} \ar[r]^{w_i} & M \ar[d]^{f} \ar[r] & 0\\
0 \ar[r] & S_{y_i} \ar[r] & E_i' \ar[r]^{v_i'} & N \ar[r] & 0}$$
where the restriction of the top row to $\Theta$ splits since
$M_{\Theta}$ is projective. Since $y_i \in \Sigma'$ and every arrow
attached to $y_i$ and supporting $E_i'$ (hence $F_i$) is entirely
contained in $\Theta$, we see that the top row splits. Let $w_i' : M
\to F_i$ such that $w_iw_i' = \id_M$. We have $f = v_i'u_iw_i'$
where $v_i'$ is not a retraction.
Hence, $u_iw_i'$ is a section since $f$ is irreducible in $\rep(Q)$. 
For $i \in I$, let $\Sigma(i) = \Sigma \backslash \{y_i\}$. Since
$(w_i')_{\Sigma(i)}$ and $(v_i')_{\Sigma(i)}$ are isomorphisms, we
see that $f_{\Sigma(i)} \cong (u_i)_{\Sigma(i)}$ is a section for
all $i \in I$. In particular, for $i \in I$, $y_i \in \Sigma$ since
$f$ is not a section. Hence, there exists a sequence
$$\Sigma \backslash \Theta = \Delta(0) \subset \Delta(1) \subset \Delta(2) \subset \dots$$
of predecessor-closed subquivers of $\Sigma$ such that $\Delta(0)$
is co-finite in $\Delta(j)$ for all $j \ge 1$ and the union of the
$\Delta(j)$ is $\Sigma$.  Moreover, for $j \ge 0$, there exists $i_j
\in I$ such that $y_{i_j}$ is not in $\Delta(j)$. By what we have
shown, $f_{\Delta(j)}$ is a section for all $j$.

For $j \ge 0$, set $M_j = M_{\Delta(j)}$, $N_j = N_{\Delta(j)}$ and
$f_j = f_{\Delta(j)}$. Observe that $M_j,N_j \in \rep^-(Q)$ for all
$j$. Let $V_j$ denote the subspace of $\Hom(N_j,M_j)$ of the
morphisms $g$ for which $gf_j$ is a multiple of $\id_{M_j}$. Observe
that $V_j$ is finite dimensional. A morphism $g \in V_j$ for which
$gf_j=\id_{M_j}$ is called normalized. By assumption, each $V_j$
contains a normalized morphism and hence is non-zero. Then one has a
non-zero map
$$g_j : V_{j+1} \to V_j$$
which is the restriction to $\Delta(j)$ and sends a normalized map
to a normalized one.

By assumption, each $V_j$ contains a normalized map $v_j : N_j \to
M_j$ such that
$$(v_j)_{\Delta(l)} = g_{l}\cdots g_{j-1}(v_j)$$ is
normalized in $V_{l}$ for $0 \le l < j$.  Using a similar argument
as in the proof of Lemma \ref{lemmasection},
we can choose $u_j \in V_j$, for $j \ge 0$, such that
$u_jf_j=\id_{M_j}$ and $(u_{j+1})_{\Delta(j)}=u_j$. Since $\Sigma$
is the union of the $\Delta(j)$, we see that $f$ is a section, a
contradiction.  This shows that $N_{\Theta}$ is finitely presented,
and so is $N_{\Omega}$.  Since $N_{\Sigma \backslash \Omega} =
M_{\Sigma \backslash \Omega}$, $N_{\Sigma \backslash \Omega}$ is
finitely co-presented.  Since $M$ is a sub-representation of
$N=M+L$, we see that
$$0 \to N_{\Omega} \to N \to N_{\Sigma
\backslash \Omega} \to 0$$ is finite, showing that $N \in \rrep(Q)$.
\qedhere

\end{proof}

\begin{Lemma}
Let $f: M \to N$ be an irreducible monomorphism in $\rep(Q)$ with
$M$ indecomposable and $N \in \rrep(Q)$.  Then $M \in \rrep(Q)$.
\end{Lemma}

\begin{proof} By Corollary \ref{Corirred}, we can assume that $f$ is the inclusion and $N =
M+L$ with $L$ finitely generated.  Let $\Sigma$ be the support of
$N$ and $\Sigma_P$ be a successor-closed subquiver of $\Sigma$ such
that $N_{\Sigma_P}$ is projective in $\rep^+(Q)$ and
$N/N_{\Sigma_P}$ is finitely co-presented.  Since $N$ is a finite
extension-representation of $N/N_{\Sigma_P}$ by $N_{\Sigma_P}$,
there exists a co-finite successor-closed subquiver $\Omega$ of
$\Sigma_P$ such that every arrow supporting $N$ with an endpoint in
$\Omega$ lies entirely in $\Sigma_P$. Being a sub-representation of
$N_{\Sigma_P}$, $M_{\Sigma_P}$ is projective. If $M_{\Sigma_P}$ is
not finitely generated, then there exists a vertex $x$ in $\Omega$
such that $P_x$ is a direct summand of $M_{\Sigma_P}$.  Since every
arrow supporting $M$ and attached to ${\rm supp}(P_x) \subseteq
\Omega$ lies in $\Sigma_P$, we see that $P_x$ is a direct summand of
$M$, giving $M=P_x$, a contradiction. Hence, $M_{\Sigma_P}$ is
finitely generated and being projective, is finitely presented.
Since $N=M+L$ where $L$ is finitely generated, there exists a
co-finite predecessor-closed subquiver $\Sigma'$ of $\Sigma
\backslash \Sigma_P$ such that $M_{\Sigma'} = N_{\Sigma'}$ is
finitely co-presented.  Thus, $M/M_{\Sigma_P}$ is also finitely
co-presented. The extension
$$0 \to M_{\Sigma_P} \to M \to M/M_{\Sigma_P} \to 0$$
is finite since
$$0 \to N_{\Sigma_P} \to N \to N/N_{\Sigma_P} \to 0$$
is finite.
\end{proof}

The preceding lemmas with their dual versions and Proposition
\ref{irred} then give the following interesting result.

\begin{Prop}
Every irreducible morphism $M \to N$ between indecomposable objects
in $\rep(Q)$ with one of $M,N$ in $\rrep(Q)$ lies entirely in
$\rrep(Q)$.  In particular, either $M \in \rep^-(Q)$ or $N \in
\rep^+(Q)$.
\end{Prop}

Therefore, we have proven the promised main result of this section.

\begin{Theo}
Any Auslander-Reiten component of $\Ga_{\rreps(Q)}$ is a connected
component of $\Ga_{\rep(Q)}$.
\end{Theo}

By Proposition \ref{propAR}, any other connected component of
$\Ga_{\rep(Q)}$ is such that the translation $\tau$ is nowhere
defined.

\begin{Exam}
Let $Q$ be the quiver
$$\xymatrixcolsep{10pt}\xymatrixrowsep{2pt}\xymatrix{& 5\ar@{.}[dl] \ar[dr] && 3\ar[dl] \ar[dr] && 1\ar[dl] \ar[dr]
& \\
&& 4 && 2 && 0}$$ of the last example. For $i \ge 0$, let $M_i$ be
the indecomposable representation of $\rep(Q)$ such that $M(j)=k$
for all $j \ge i$ and $M(j)=0$, otherwise. Since $Q$ has no infinite
path, $\rrep(Q) = \rep^b(Q)$ are the finite dimensional
representations. The only indecomposable infinite dimensional
representations of $\rep(Q)$, up to isomorphisms, are the $M_i$. The
only connected component of $\Ga_{\rep(Q)}$ which is not a connected
component of $\Ga_{\rreps(Q)}$ is the following component with
trivial valuations:
$$\xymatrixrowsep{10 pt}\xymatrixcolsep{10 pt}\xymatrix{\cdots\ar[r] & M_4 \ar[r] &  M_2 \ar[r]  & M_0\ar[r]
 &  M_1 \ar[r] & M_3\ar[r] & M_5 \ar[r]   & \cdots \\
}
 $$
\end{Exam}

We end this paper with the following conjecture.

\begin{Conj}
Let $Q$ be a strongly locally finite quiver.  The connected
components of $\Ga_{\rep(Q)}$ containing representations not in
$\rrep(Q)$ are connected subquivers of the quiver
$$\cdots \to \circ \to \circ \to \circ \to \cdots$$
and all have trivial valuations.
\end{Conj}

\end{document}